\documentclass[12pt]{amsart}
\pdfoutput=1

\usepackage[paperheight=11in, paperwidth=8.5in, left=1in, top=1.25in, right=1in, bottom=1.25in]{geometry}
\usepackage[linktocpage=true, linktoc=all, colorlinks=true, linkcolor=black, citecolor=black, filecolor=black, urlcolor=black, pagebackref=false, pdfstartpage={1}, pdfstartview={FitH}, pdftitle={}, pdfauthor={}, pdfsubject={}, pdfcreator={}, pdfproducer={}, pdfkeywords={}]{hyperref}
\usepackage{amsmath}
\usepackage{amsxtra}
\usepackage{amscd}
\usepackage{amsthm}
\usepackage{amsfonts}
\usepackage{amssymb}
\usepackage{epsfig}
\usepackage{bbm}
\usepackage{bm}
\usepackage{graphics}
\usepackage{graphicx}
\usepackage{accents}
\usepackage{dynkin-diagrams}
\usepackage{comment}
\usepackage{enumitem}
\usepackage{mathrsfs}
\usepackage{mathtools}
\usepackage{tikz-cd}
\usepackage{pgf}
\usetikzlibrary{arrows, matrix}
\usepackage[aligntableaux=center]{ytableau}
\usepackage{here}
\usepackage{simplewick}
\usepackage[capitalize,nameinlink,noabbrev,compress]{cleveref}
\usepackage[justification=centering]{caption}
\usepackage{cancel}

\restylefloat{figure}
\makeatletter
\newcommand{\pushright}[1]{\ifmeasuring@#1\else\omit\hfill$\displaystyle#1$\fi\ignorespaces}
\newcommand{\pushleft}[1]{\ifmeasuring@#1\else\omit$\displaystyle#1$\hfill\fi\ignorespaces}
\makeatother

\numberwithin{equation}{section}
\newtheorem{thm}{Theorem}[section]
\newtheorem{prop}[thm]{Proposition}
\newtheorem{lem}[thm]{Lemma}

\newtheorem{cor}[thm]{Corollary}

\theoremstyle{definition}
\newtheorem{example}[thm]{Example}
\AtBeginEnvironment{example}{%
  \pushQED{\qed}%
}
\AtEndEnvironment{example}{\popQED\endexample}
\newtheorem{rem}[thm]{Remark}
\AtBeginEnvironment{rem}{%
  \pushQED{\qed}%
}
\AtEndEnvironment{rem}{\popQED\endexample}

\crefname{thm}{Theorem}{Theorems}
\crefname{prop}{Proposition}{Propositions}
\crefname{lem}{Lemma}{Lemmas}
\crefname{cor}{Corollary}{Corollaries}
\crefname{conj}{Conjecture}{Conjectures}
\crefname{dfn}{Definition}{Definitions}
\crefname{example}{Example}{Examples}
\crefname{rem}{Remark}{Remarks}

\newcommand{\C}{{\mathbb C}}

\newcommand{\gl}{\mathfrak{gl}}

\newcommand{\End}{\mathop{\rm End}}

\newcommand{\Tr}{{\rm Tr}}

\newcommand{\be}{\begin{equation*}}
\newcommand{\ee}{\end{equation*}}

\newcommand{\betheh}{{\mathcal{B}_h}}
\newcommand{\conjugate}[1]{#1^*}
\newcommand{\cyc}{\mathsf{cyc}}
\newcommand{\dd}{\mathsf{d}}
\newcommand{\DD}{\mathsf{D}}
\newcommand{\defn}[1]{{\bfseries\itshape #1}}
\DeclareMathOperator{\Diag}{\mathsf{Diag}}
\newcommand{\dop}{\mathcal{D}}
\newcommand{\du}[2]{\partial_{#1}^{\hspace*{1pt}#2}}
\newcommand{\entire}{{\mathcal O}}
\renewcommand{\eqref}[1]{\hyperref[#1]{\textup{(\ref*{#1})}}}
\newcommand{\eval}[1]{\hspace*{-1pt}\big\rvert_{#1}\hspace*{2pt}}

\newcommand{\evec}[1]{\mathsf{e}_{#1}}
\newcommand{\ff}[1]{\mathsf{f}^{#1}}
\newcommand{\hh}{\mathsf{g}}
\DeclareMathOperator{\GL}{GL}
\DeclareMathOperator{\Gr}{Gr}

\newcommand{\SG}[1]{\mathfrak{S}_{#1}}
\newcommand{\Sid}{\mathbbm{1}}

\newcommand{\x}{\mathbf{x}}
\DeclareMathOperator{\Wr}{Wr}

\begin{document}
\title[Spaces of quasi-exponentials]{Positivity and universal Pl\"{u}cker coordinates for spaces of quasi-exponentials}

\author{Steven N. Karp, Evgeny Mukhin, and Vitaly Tarasov}
\address{S.N.K.: Department of Mathematics, University of Notre Dame,
255 Hurley Hall,
Notre Dame, IN 46556, USA}
\email{\href{mailto:skarp2@nd.edu}{skarp2@nd.edu}}
\address{E.M.: Department of Mathematical Sciences,
Indiana University Indianapolis,
402 N. Blackford St., LD 270, 
Indianapolis, IN 46202, USA}
\email{\href{mailto:emukhin@iu.edu}{emukhin@iu.edu}}
\address{V.T.: Department of Mathematical Sciences,
Indiana University Indianapolis,
402 N. Blackford St., LD 270, 
Indianapolis, IN 46202, USA}
\email{\href{mailto:vtarasov@iu.edu}{vtarasov@iu.edu}}

\thanks{S.N.K.\ is partially supported by the National Science Foundation under Award No.\ 2452061 and by a travel support gift from the Simons Foundation. E.M.\ is partially supported by the Simons Foundation grant \#709444. V.T.\ is partially supported by the Simons Foundation grants \#430235, \#852996.}


\begin{abstract}
A quasi-exponential is an entire function of the form $e^{cu}p(u)$, where $p(u)$ is a polynomial and $c \in \mathbb{C}$. Let $V = \langle e^{h_1u}p_1(u), \dots, e^{h_Nu}p_N(u) \rangle$ be a vector space with a basis of quasi-exponentials. We show that if $h_1, \dots, h_N$ are nonnegative and all of the complex zeros of the Wronskian $\operatorname{Wr}(V)$ are real, then $V$ is totally nonnegative in the sense that all of its Grassmann--Pl\"{u}cker coordinates defined by the Taylor expansion about $u=t$ are nonnegative, for any real $t$ greater than all of the zeros of $\operatorname{Wr}(V)$. Our proof proceeds by showing that the higher Gaudin Hamiltonians $T_\lambda^G(t)$ introduced in \cite{alexandrov_leurent_tsuboi_zabrodin14} are universal Pl\"{u}cker coordinates about $u=t$ for the Wronski map on spaces of quasi-exponentials. The result that $V$ is totally nonnegative follows from the fact that $T_\lambda^G(t)$ is positive semidefinite, which we establish using partial traces. We also show that if $h_1 = \cdots = h_N = 0$ then $T_\lambda^G(t)$ equals $\beta^\lambda(t)$, which is the universal Pl\"{u}cker coordinate for the Wronski map on spaces of polynomials introduced in \cite{karp_purbhoo}.
\end{abstract}

\dedicatory{Dedicated to Frank Sottile on the occasion of his 60th birthday}
\maketitle

\section{Introduction}\label{sec_introduction}

\noindent The \defn{Wronskian} of $N$ entire functions $f_1, \dots, f_N : \C \to \C$ in the variable $u$ is
\begin{align}\label{wronskian_formula}
\Wr(f_1, \dots, f_N) := \det\hspace*{-1pt}\big(\du{u}{i-1}f_j\big)_{1 \le i,j \le N} = \det\begin{bmatrix}
f_1 & \cdots & f_N \\[2pt]
f_1' & \cdots & f_N' \\
\vdots & \ddots & \vdots \\[2pt]
f_1^{(N-1)} & \cdots & f_N^{(N-1)}
\end{bmatrix}.
\end{align}
It is identically zero if $f_1, \dots, f_N$ are linearly dependent; otherwise, it only depends on the linear span $V$ of $f_1, \dots, f_N$ up to rescaling. In particular, it makes sense to write $\Wr(V)$ and consider its zeros. In 1993, Boris and Michael Shapiro 
conjectured a remarkable reality property for Wronskians of spaces of polynomials, which was proved in \cite{mukhin_tarasov_varchenko09a}: 
\begin{thm}[Reality theorem for spaces of polynomials \cite{mukhin_tarasov_varchenko09a}]\label{MTV_SS}
Suppose that $V \subseteq \C[u]$ is a finite-dimensional vector space of polynomials. If all complex zeros of $\Wr(V)$ are real, then $V$ is real (i.e.\ it has a basis of polynomials with real coefficients).\hfill\qed
\end{thm}

The Shapiro--Shapiro conjecture has attracted a lot of attention in algebraic geometry (cf.\ \cite{sottile11}), since it can be equivalently formulated as the statement that any Schubert intersection problem on a Grassmannian $\Gr(N,\C^m)$ defined by osculating flags to a rational normal curve at real points has all real solutions.

\cref{MTV_SS} was generalized in \cite{mukhin_tarasov_varchenko09c,mukhin_tarasov_varchenko08} to spaces of \defn{quasi-exponentials}, i.e., spaces of the form $V = \langle e^{h_1u}p_1(u), \dots, e^{h_Nu}p_N(u)\rangle$, where $h_1, \dots, h_N$ are complex numbers (called the \defn{exponents} of $V$) and $p_1(u), \dots, p_N(u)$ are polynomials. Note that the Wronskian of $V$ is also a quasi-exponential:
$$
\Wr(V) = e^{(h_1 + \cdots + h_N)u}g(u), \quad \text{ for some polynomial } g(u).
$$

\begin{thm}[Reality theorem for spaces of quasi-exponentials {\cite{mukhin_tarasov_varchenko09c,mukhin_tarasov_varchenko08}}]\label{MTV_quasiexp}
Suppose that $V = \langle e^{h_1u}p_1(u), \dots, e^{h_Nu}p_N(u)\rangle$ is a space of quasi-exponentials. If the exponents of $V$ are real and all complex zeros of $\Wr(V)$ are real, then $V$ is real.\hfill\qed
\end{thm}

Observe that taking $h_1 = \cdots = h_N = 0$ in \cref{MTV_quasiexp} recovers \cref{MTV_SS}.
A generalization of \cref{MTV_SS} in a different direction upgrades reality to positivity \cite{karp_purbhoo}. Namely, let $\Gr(N,X)$ denote the \defn{Grassmannian} of all $N$-dimensional subspaces of the vector space $X$. Given a space $V = \langle f_1, \dots, f_N\rangle\in\Gr(N,\entire)$ of entire functions, consider their Taylor expansions about $u=t$ (for some $t\in\C$):
$$
f_j(u) = \sum_{i=0}^\infty\frac{f_{i,j}}{i!}(u-t)^i \quad \text{ for } 1 \le j \le N, \text{ where } f_{i,j} =\du{u}{i}f_j(t)\in\C.
$$
The \defn{Pl\"{u}cker coordinates} of $V$ about $u=t$ are the $N\times N$ minors of the infinite matrix
\begin{align}\label{A_matrix}
\begin{bmatrix}
f_{0,1} & \cdots & f_{0,N} \\[2pt]
f_{1,1} & \cdots & f_{1,N} \\
\vdots &  & \vdots
\end{bmatrix},
\end{align}
which are well-defined up to simultaneous rescaling. They are naturally labeled by partitions $\lambda = (\lambda_1 \ge \cdots \ge \lambda_N \ge 0)$ with length $\ell(\lambda) \le N$, where $\lambda$ indexes the minor using rows $\lambda_N, \lambda_{N-1} + 1, \dots, \lambda_1 + N-1$.

We call $V\in\Gr(N,\entire)$ \defn{totally nonnegative} (respectively, \defn{totally positive}) about $u=t$ if all such Pl\"{u}cker coordinates are real and nonnegative (respectively, positive), up to rescaling. This is an analogue for $\Gr(N,\entire)$ of the notion of total positivity for finite-dimensional Grassmannians $\Gr(N,\C^M)$ introduced by Lusztig \cite{lusztig94} and Postnikov \cite{postnikov06}. A connection with Wronskians of polynomials was established in \cite{karp_purbhoo}:

\begin{thm}[Positivity theorem for spaces of polynomials {\cite[Theorem 1.10(i)]{karp_purbhoo}}]\label{KP_positivity}
Suppose that $V \in \Gr(N,\C[u])$ is a space of polynomials, and write $\Wr(V) = (u+z_1) \cdots (u+z_n)$. If $z_1, \dots, z_n$ are real, then $V$ is totally nonnegative about $u=t$ for all $t \ge -z_1, \dots, -z_n$.\hfill\qed
\end{thm}

There is also a totally positive version of \cref{KP_positivity}; see \cite[Theorem 1.10(ii)]{karp_purbhoo}. \cref{KP_positivity} verifies a conjecture posed by the second and third authors in 2017 as well as independently in \cite{karp24}. One motivation for \cref{KP_positivity} is an application to real Schubert calculus: it implies that any Schubert intersection problem on a Grassmannian $\Gr(N,\C^M)$ defined by non-overlapping secant flags to a rational normal curve at real points has all real solutions. This verifies the secant conjecture of Sottile in the case of divisor Schubert conditions; see \cite[Section 1.3]{karp_purbhoo} for further discussion.

In this paper we generalize and provide a new proof of \cref{KP_positivity}:
\begin{thm}[Positivity theorem for spaces of quasi-exponentials]\label{positivity}
Suppose that $V = \langle e^{h_1u}p_1(u), \dots, e^{h_Nu}p_N(u)\rangle\in\Gr(N,\entire)$ is a space of quasi-exponentials, and write
\begin{align}\label{positivity_wronskian}
\Wr(V) = e^{(h_1 + \cdots + h_N)u}(u + z_1) \cdots (u + z_n).
\end{align}
Suppose that $z_1, \dots, z_n$ are real.
\begin{enumerate}[label=(\roman*), leftmargin=*, itemsep=2pt]
\item\label{positivity_weak} If $h_1, \dots, h_N \ge 0$, then $V$ is totally nonnegative about $u=t$ for all real $t \ge -z_1, \dots, -z_n$.
\item\label{positivity_strict} If $h_1, \dots, h_N > 0$, then $V$ is totally positive about $u=t$ for all real $t > -z_1, \dots, -z_n$.
\end{enumerate}
\end{thm}

\begin{example}\label{eg_positivity}
We illustrate \cref{positivity} for the space $V := \langle e^{h_1u}, \dots, e^{h_Nu}\rangle$ of pure exponentials, where $h_1, \dots, h_N\in\C$ are distinct. We have (up to rescaling)
$$
\Wr(V) = e^{(h_1 + \cdots + h_N)u}.
$$
For $t\in\C$, the matrix from \eqref{A_matrix} equals
$$
\begin{bmatrix}
e^{h_1t} & \cdots & e^{h_Nt} \\
h_1e^{h_1t} & \cdots & h_Ne^{h_Nt} \\
h_1^2e^{h_1t} & \cdots & h_N^2e^{h_Nt} \\
\vdots & & \vdots
\end{bmatrix}.
$$
Then the Pl\"{u}cker coordinate of $V$ about $u=t$ indexed by the partition $\lambda$ equals (up to a global scalar depending on $t$) the Schur polynomial $s_\lambda(h_1, \dots, h_N)$. Thus part \ref{positivity_weak} is equivalent to the statement that $s_\lambda(h_1, \dots, h_N) \ge 0$ if $h_1, \dots, h_N \ge 0$, and part \ref{positivity_strict} is equivalent to $s_\lambda(h_1, \dots, h_N) > 0$ if $h_1, \dots, h_N > 0$. Both statements follow from the well-known fact that $s_\lambda(h_1, \dots, h_N)$ expands into monomials with positive coefficients.
\end{example}

We provide two proofs of \cref{positivity}. The first uses the \defn{higher Gaudin Hamiltonians} $T_\lambda(u)\in\End((\C^N)^{\otimes n})$ introduced by Alexandrov, Leurent, Tsuboi, and Zabrodin \cite{alexandrov_leurent_tsuboi_zabrodin14}, indexed by partitions $\lambda$ with $\ell(\lambda)\le N$. They are appropriate limits of the Yangian transfer matrices corresponding to the evaluation module of the Yangian obtained from the irreducible $\mathfrak{gl}_N$-module of highest weight $\lambda$, acting in a tensor product of evaluation vector representations with evaluation parameters $-z_i$. In particular, the higher Gaudin Hamiltonians all commute: $T_\lambda(u)T_\mu(v)=T_\mu(v)T_\lambda(u)$. Explicitly, we have
\begin{align}\label{defn_Tlambda}
T_\lambda(u) := (u + z_1 + \dd_1) \cdots (u + z_n + \dd_n)s_\lambda(h) \in \End((\C^N)^{\otimes n}),
\end{align}
where $h$ is an $N\times N$ matrix with eigenvalues $h_1, \dots, h_N$, $s_\lambda(h)$ is the Schur polynomial evaluated at the eigenvalues of $h$, and $\dd_k$ is the usual matrix derivative with respect to the transpose of $h$ acting in the $k$th tensor factor of $(\C^N)^{\otimes n}$ (see \cref{sec_Tlambda}). 

We show that for all $t\in\C$, the operators $T_\lambda(t)$ are universal Pl\"{u}cker coordinates for the Wronski map on spaces of quasi-exponentials about $u=t$ (see \cref{quasiexp_coordinates}). This means that for every space $V = \langle e^{h_1u}p_1(u), \dots, e^{h_Nu}p_N(u)\rangle$ with Wronskian \eqref{positivity_wronskian}, there exists a common eigenspace $E \subseteq (\C^N)^{\otimes n}$ of the operators $T_\lambda(t)$, such that by replacing each $T_\lambda(t)$ by its eigenvalue, we obtain the vector of Pl\"{u}cker coordinates of $V$ about $u=t$. We accomplish this by proving a Jacobi--Trudi identity expressing an arbitrary Pl\"{u}cker coordinate as a determinant of Pl\"{u}cker coordinates indexed by single-column partitions (see \eqref{equation_JT_2}), and use the fact from \cite{alexandrov_leurent_tsuboi_zabrodin14} that the operators $T_\lambda(t)$ satisfy the same Jacobi--Trudi identity.

With this result in hand, \cref{positivity} then follows by proving that $T_\lambda(t)$ is positive (semi)definite for appropriate values of the parameters (see \cref{Tlambda_positive}). We accomplish this by deriving a formula for $T_\lambda(t)$ in terms of partial traces and not involving any derivatives (see \cref{Tlambda_formula}). Our arguments also provide a new proof of \cref{KP_positivity}.

Our second proof of \cref{positivity} deduces it directly from \cref{KP_positivity}. The idea is to prove part \ref{positivity_weak} of \cref{positivity} by writing each exponential function as a limit of polynomials (e.g.\ $e^u = \lim_{m\to\infty}\big(1 + \frac{u}{m}\big)^m$). We then show that part \ref{positivity_weak} implies part \ref{positivity_strict}, using the fact that the linear operator $e^{cu}$ on $\entire$ (where $c > 0$) is represented by a triangularly totally positive matrix.

In summary, we give two proofs of \cref{positivity}, one using the theory of Gaudin models and the other using the results of \cite{karp_purbhoo}. The advantage of the former proof is that it is short and conceptual for experts in integrable systems, and it also provides a new proof of \cref{KP_positivity}. The disadvantage is that it relies on results in \cite{alexandrov_leurent_tsuboi_zabrodin14}, which the proof in \cite{karp_purbhoo} does not. Rather, \cite{karp_purbhoo} uses the combinatorics of permutations and symmetric functions and the representation theory of symmetric groups.

\subsection*{Outline}
In \cref{sec_background} we recall some background. Importantly, in \cref{equality_columns}, we express the operators $T_{(1^i)}(u)$ (indexed by single-column partitions) in terms of the coefficients of the universal differential operator $\dop^\betheh$ from \cite{mukhin_tarasov_varchenko06}. We also recall from \cite{alexandrov_leurent_tsuboi_zabrodin14} that an arbitrary $T_\lambda(u)$ can be expressed in terms of the $T_{(1^i)}(u)$'s by a dual Jacobi--Trudi formula. In \cref{sec_JT}, we show that the Pl\"{u}cker coordinates of $V$ satisfy the same formula. In \cref{sec_universal} we use these facts, along with results about $\dop^\betheh$ from \cite{mukhin_tarasov_varchenko06}, to prove that the operators $T_\lambda(t)$ are universal Pl\"{u}cker coordinates about $u=t$ for the Wronski map on spaces of quasi-exponentials.

In \cref{sec_formula_proof} we prove a formula for $T_\lambda(u)$ in terms of partial traces, which we use in \cref{sec_semidefinite} to show that it is positive semidefinite for appropriate choices of the parameters. We give the first proof of our main result \cref{positivity} at the end of \cref{sec_semidefinite}. In \cref{sec_poly_to_quasi} we give our second proof of \cref{positivity}, by deducing it from \cref{KP_positivity}.

Finally, in \cref{sec_quasi_to_poly}, we explain how our paper is connected to the paper \cite{karp_purbhoo}, in which \cref{KP_positivity} was established. In particular, we show that when $h_1 = \cdots = h_N = 0$, the operators $T_\lambda(u)$ coincide with the operators $\beta_\lambda(u)$ introduced in \cite{karp_purbhoo}, which are universal Pl\"{u}cker coordinates for the Wronski map on spaces of polynomials.

\subsection*{Notation}
For the benefit of readers comparing our work with the literature, we briefly clarify how our notation differs from \cite{mukhin_tarasov_varchenko08,alexandrov_leurent_tsuboi_zabrodin14,karp_purbhoo}. Our $z_k$ is denoted by $-b_k$ in \cite{mukhin_tarasov_varchenko08} and by $-x_k$ in \cite{alexandrov_leurent_tsuboi_zabrodin14}. Our $h_i$ is denoted by $K_i$ in \cite{mukhin_tarasov_varchenko08} and by $k_i$ in \cite{alexandrov_leurent_tsuboi_zabrodin14}. Our $T_\lambda(u)$ is denoted by $T_\lambda^G(u)$ in \cite{alexandrov_leurent_tsuboi_zabrodin14}; we emphasize that the same paper also considers the operator $\mathsf{T}_\lambda^G(u)$ (which is a rescaling of $T_\lambda^G(u)$) and the operator $\mathsf{T}_\lambda(u)$ (which is a Yangian analogue of $\mathsf{T}_\lambda^G(u)$). Our $N$ is denoted by $d$ in \cite{karp_purbhoo}, and therein $\lambda$ appears in superscripts rather than subscripts. 

\subsection*{Acknowledgments}
We thank John Harnad for initiating discussions about the relationship between \cite{alexandrov_leurent_tsuboi_zabrodin14} and \cite{karp_purbhoo}, and Kevin Purbhoo for helpful conversations.

\section{Background}\label{sec_background}

\noindent In this section, we set some conventions for notation and for Pl\"{u}cker coordinates of spaces of functions. We then recall some standard background on symmetric functions, symmetric group characters, and partial traces; see, e.g., \cite[Chapter 7]{stanley24}, \cite{sagan01}, \cite[Section 2.4.3]{nielsen_chaung00} for further details. We also recall the higher Gaudin Hamiltonians $T_\lambda(u)$ introduced in \cite{alexandrov_leurent_tsuboi_zabrodin14}, and the universal differential operator $\dop^\betheh$ studied in \cite{mukhin_tarasov_varchenko06}.

\subsection{Notation}
We fix $z_1, \dots, z_n \in \C$, and set $z_K := \prod_{k\in K}z_k$ for $K\subseteq [n] := \{1, \dots, n\}$. We let $(x)_m := x(x-1)\cdots(x-m+1)$ denote the falling factorial. Let $\entire$ denote the space of entire functions on $\C$, and let $\partial_x := \frac{\partial}{\partial x}$ denote the partial derivative with respect to $x$.

A \defn{partition} of $m\ge 0$ is a weakly decreasing sequence of nonnegative integers $\lambda = (\lambda_1, \dots, \lambda_l)$ such that $\lambda_1 + \cdots + \lambda_l = m$. We write $|\lambda| = m$ or $\lambda\vdash m$. The number of nonzero parts of $\lambda$ is called its \defn{length}, denoted $\ell(\lambda)$. We identify partitions up to trailing zeros (e.g.\ $(4,1,1) = (4,1,1,0,0)$). We define the \defn{diagram} of $\lambda$ to be the array of $|\lambda|$ left-justified boxes with $\lambda_i$ boxes in the $i$th row from the top for all $i\ge 1$. The \defn{conjugate} $\conjugate{\lambda}$ of $\lambda$ is the partition whose diagram is the transpose of that of $\lambda$. We let $(1^m) := \conjugate{(m)} = (1, \dots, 1)$ denote the partition of $m$ whose diagram is a single column of length $m$, which is the conjugate of the partition $(m)$ whose diagram is a single row of length $m$. In particular, we let $\varnothing := (0)=(1^0)$ denote the unique partition of $0$.

All vector spaces $V$ are defined over $\mathbb{C}$. For $S\subseteq V$, we let $\langle S\rangle$ denote the linear span of $S$. For $A,B\in\End(V)$, we let $AB$ denote the composition $A\circ B$. We let $\evec{1}, \dots, \evec{N}$ denote the standard basis of $\C^N$. We identify $\End(\C^N)$ with the space of $N\times N$ matrices over $\C$. For $1 \le i, j \le N$, we let $E_{i,j} := \evec{i}\evec{j}^T\in\End(\C^N)$ denote the matrix whose $(i,j)$-entry is $1$ and whose other entries are $0$. Also, we let $h$ denote an element of $\End(\C^N)$ with entries $h_{i,j}$ for $1 \le i,j \le N$ and eigenvalues $h_1, \dots, h_N\in\C$. More generally, for any matrix $A$ we let $A_{I,J}$ denote the submatrix of $A$ restricted to the row subset $I$ and the column subset $J$.

We will frequently work with the vector space $(\C^N)^{\otimes L}$, the $|L|$-fold tensor product of $\C^N$ with factors indexed by $L$. For $K\subseteq L$, we regard $\End((\C^N)^{\otimes K})$ as a subalgebra of $\End((\C^N)^{\otimes L})$ acting as the identity on the tensor factors in $L\setminus K$. Let $h^{(K)}\in\End((\C^N)^{\otimes K})$ denote the diagonal action of $h$ on the tensor factors in $K$.

\subsection{Pl\"{u}cker coordinates and translation}
Let $V\in\Gr(N,\entire)$ be a space of entire functions. Recall the Pl\"{u}cker coordinates of $V$ about $u=t$ defined in \eqref{A_matrix}. We will prefer to instead view such coordinates as the Pl\"{u}cker coordinates of the translated space $V(t)\in\Gr(N,\entire)$ about $u=0$; this will allow us to assume that all of our Taylor expansions are about $u=0$, unless specified otherwise.

Explicitly, fix a basis $(f_1, \dots, f_N)$ of $V$. For every partition $\lambda$ with $\ell(\lambda) \le N$, we define the \defn{Pl\"{u}cker coordinate} (about $u=0$)
$$
\Delta_\lambda(V) := \det(A_{\{\lambda_N, \lambda_{N-1} + 1, \dots, \lambda_1 + N-1\}, \{1, \dots, N\}}) = \det\hspace*{-1pt}\big(\du{u}{\lambda_{N-i+1}+i-1}f_j(0))_{1 \le i,j \le N},
$$
where $A$ denotes the matrix in \eqref{A_matrix}; it is implicit in the notation $\Delta_\lambda(V)$ that we are taking Taylor expansions about $u=0$. We also set $\Delta_\lambda(V) := 0$ if $\ell(\lambda) > N$. The vector of Pl\"{u}cker coordinates $(\Delta_\lambda(V))_\lambda$ is well-defined (i.e.\ independent of the choice of basis) up to rescaling.

Given $t\in\C$, we define the \defn{translation} of $V$ by $t$ to be
$$
V(t) := \{f(u+t) : f(u)\in V\}\in\Gr(N,\entire),
$$
so that $V = V(0)$. Then for all partitions $\lambda$ with $\ell(\lambda)\le N$, we have
\begin{align}\label{plucker_translation}
\Delta_\lambda(V(t)) = \det\hspace*{-1pt}\big(\du{u}{\lambda_{N-i+1}+i-1}f_j(t)\big)_{1 \le i,j \le N} = (-1)^{\binom{N}{2}}\det\hspace*{-1pt}\big(\du{u}{\lambda_i+N-i}f_j(t)\big)_{1 \le i,j \le N}.
\end{align}
That is, the Pl\"{u}cker coordinates of $V(t)$ about $u=0$ are precisely the Pl\"{u}cker coordinates of $V$ about $u=t$. 

Note that \eqref{plucker_translation} allows us to treat $\Delta_\lambda(V(t))$ as an entire function of $t$. For example, by \eqref{wronskian_formula} we have
\begin{align}\label{equation_wronskian_empty}
\Delta_\varnothing(V(u)) = \Wr(V).
\end{align}
Also, by \cite[(2.14)]{karp_purbhoo}, we have the explicit formula (up to rescaling)
\begin{align}\label{translation_identity}
\Delta_\mu(V(u)) = \sum_{\lambda\supseteq\mu}\frac{\ff{\lambda/\mu}}{(|\lambda| - |\mu|)!}u^{|\lambda| - |\mu|}\Delta_\lambda(V) \quad \text{ for all partitions } \mu,
\end{align}
where the sum is over all partitions $\lambda$ whose diagram contains $\mu$, and $\ff{\lambda/\mu}$ is the number of standard Young tableaux of skew shape $\lambda/\mu$ (we will not need the precise definition).

\subsection{Symmetric functions and symmetric group characters}
Let $\x := \{x_i\}$ denote a finite or countably infinite set of variables, and let $\Lambda(\x)$ denote the $\C$-algebra of symmetric functions in $\x$. Given a partition $\lambda$, we let $s_\lambda(\x) \in \Lambda(\x)$ denote the \defn{Schur function}; if $\x = \{x_1, \dots, x_N\}$ is finite, then $s_\lambda(x_1, \dots, x_N)\neq 0$ if and only if $\ell(\lambda)\le N$, and
$$
s_\lambda(x_1, \dots, x_N) = \frac{\det(x_j^{\lambda_i + N - i})_{1 \le i,j \le N}}{\det(x_j^{N - i})_{1 \le i,j \le N}} \quad \text{ if } \ell(\lambda) \le N.
$$
We also define the \defn{power sum} $p_\lambda(\x) \in \Lambda(\x)$ by
$$
p_\lambda(\x) := p_{\lambda_1}(\x) \cdots p_{\lambda_{\ell(\lambda)}}(\x), \quad \text{ where } p_k(\x) := \sum_i x_i^k \text{ for all } k \ge 1.
$$
For any symmetric function $f(\x) \in \Lambda(\x)$, we let $f(h) := f(h_1, \dots, h_N)$ denote $f$ evaluated at the eigenvalues of $h\in\End(\C^N)$.

Given a finite set $K$ of size $m$, we let $\SG{K}$ denote the \defn{symmetric group} of all permutations of $K$ (i.e.\ bijections $K\to K$), and we let $\C[\SG{K}]$ denote its group algebra. Let $Z(\C[\SG{K}])$ denote the center of $\C[\SG{K}]$, which is spanned by $\sum_{\sigma\in C}\sigma$ over all conjugacy classes $C\subseteq\SG{K}$. We let $\C[\SG{K}]$ act on $(\C^N)^{\otimes K}$ by
$$
\sigma\Big(\underset{k\in K}{\otimes}v_k\Big) := \underset{k\in K}{\otimes}v_{\sigma^{-1}(k)} \quad \text{ for all } \sigma\in\SG{K} \text{ and } v_k\in\C^N
$$
and extending linearly. When $K = [m]$, we denote $\SG{K}$ by $\SG{m}$.

For $\sigma\in\SG{K}$, we let $\cyc(\sigma)$ denote its \defn{cycle type}, which is the partition of $m$ whose parts are the lengths of the cycles in the disjoint cycle decomposition of $\sigma$. For $\lambda\vdash m$, we let $\chi^\lambda : \SG{K}\to\C$ denote the character of the Specht module indexed by $\lambda$, which is an irreducible representation of $\SG{K}$ of dimension $\ff{\lambda} = \ff{\lambda/\varnothing}$. Our conventions are such that $\chi^{(m)}$ is the identity character and $\chi^{(1^m)}$ is the sign character.

We will use three additional facts about symmetric functions and symmetric group characters. Firstly, $s_\lambda(x_1, \dots, x_N)$ has a positive monomial expansion \cite[Section 7.10]{stanley24}:
\begin{align}\label{schur_monomial}
s_\lambda(x_1, \dots, x_N) \in \mathbb{N}[x_1, \dots, x_N].
\end{align}
Secondly, we have the Schur function expansion into power sums \cite[Corollary 7.17.5]{stanley24}:
\begin{align}\label{schur_expansion}
s_\lambda(\x) = \frac{1}{m!}\sum_{\sigma\in\SG{m}}\chi^\lambda(\sigma)p_{\cyc(\sigma)}(\x) \quad \text{ for all } \lambda\vdash m.
\end{align}
Thirdly, traces of matrix powers are power sums:
\begin{align}\label{power_sum_trace}
\Tr(h^k) = h_1^k + \cdots + h_N^k = p_k(h) \quad \text{ for all } k \ge 1.
\end{align}

\subsection{Partial traces}\label{sec_partial_trace}
Given finite-dimensional vector spaces $V_1$ and $V_2$, we define the \defn{partial trace} $\Tr_2 : \End(V_1 \otimes V_2) \to \End(V_1)$ by
\begin{align}\label{defn_partial_trace}
\Tr_2(A_1 \otimes A_2) := \Tr(A_2)A_1 \quad \text{ for all } A_1 \in \End(V_1) \text{ and } A_2 \in \End(V_2)
\end{align}
and extending linearly, where $\Tr(A_2)\in\C$ is the usual trace. Equivalently, if $(v_1, \dots, v_N)$ is a basis of $V_2$ with dual basis $(v_1^*, \dots, v_N^*)$, then
\begin{align}\label{partial_trace_decomposition}
\Tr_2(A) = \sum_{j=1}^N(I_{V_1}\otimes v_j^*)\hspace*{1pt}A\hspace*{1pt}(I_{V_1}\otimes v_j),
\end{align}
where $I_{V_1}\in\End(V_1)$ denotes the identity operator. We will need the following property:
\begin{lem}\label{trace_commute}
Let $A\in\End(V_1\otimes V_2)$ and $B\in\End(V_2)$. Then
$$
\Tr_2(A(I_{V_1}\otimes B)) = \Tr_2((I_{V_1}\otimes B)A).
$$
\end{lem}

\begin{proof}
By linearity, we may assume that $A = A_1\otimes A_2$. Then the result follows from \eqref{defn_partial_trace} and the fact that $\Tr(A_2B) = \Tr(BA_2)$.
\end{proof}

More generally, given finite-dimensional vector spaces $V_1, \dots, V_m$, for $K\subseteq [m]$ we define the \defn{partial trace}
$$
\Tr_K : \End(V_1 \otimes \cdots \otimes V_m) \to \End\bigg(\bigotimes_{l\in [m]\setminus K}V_l\bigg)
$$
by
\begin{align}\label{equation_partial_trace_general}
\Tr_K(A_1 \otimes \cdots \otimes A_m) := \prod_{k\in K}\Tr(A_k) \cdot \bigotimes_{l\in [m]\setminus K}A_l \quad \text{ for all } A_k\in\End(V_k)
\end{align}
and extending linearly. Note that if $K,L\subseteq [m]$ are disjoint, then $\Tr_K\Tr_L = \Tr_L\Tr_K = \Tr_{K\cup L}$. Also observe that $\Tr_{[m]}(A)\in\C$ is the usual trace of $A\in\End(V_1 \otimes \cdots \otimes V_m)$.

\subsection{Higher Gaudin Hamiltonians \texorpdfstring{$T_\lambda(u)$}{T}}\label{sec_Tlambda}
We carefully define the operators $T_\lambda(u)$ from \eqref{defn_Tlambda}, following \cite{alexandrov_leurent_tsuboi_zabrodin14}. Regarding $h\in\End(\C^N)$ as an indeterminate, for any scalar-valued function $f : \End(\C^N) \to \C$, we define the (transposed) \defn{matrix derivative}
$$
\dd(f(h)) := \sum_{1 \le i,j \le N}\partial_{\varepsilon}\eval{\varepsilon=0}f(h + \varepsilon E_{j,i}) \cdot E_{i,j} \in \End(\C^N).
$$
More generally, let $L\subseteq [m]$ and $k\in [m]\setminus L$. Then for any $f : \End(\C^N) \to \End((\C^N)^{\otimes L})$, we define
$$
\dd_k(f(h)) := \sum_{1 \le i,j \le N}\partial_{\varepsilon}\eval{\varepsilon=0}f(h + \varepsilon E_{j,i}) \cdot E_{i,j}^{(k)} \in \End((\C^N)^{\otimes L\cup\{k\}}).
$$
Note that $\dd_k\dd_l(f(h)) = \dd_l\dd_k(f(h))$ for all $k\neq l$ in $[m]\setminus L$. For $K\subseteq [m]\setminus L$, we set $\dd_K := \prod_{k\in K}\dd_k$.

Then as in \eqref{defn_Tlambda}, given a partition $\lambda$, we define
$$
T_\lambda(u) := (u + z_1 + \dd_1) \cdots (u + z_n + \dd_n)s_\lambda(h) \in \End((\C^N)^{\otimes n})
$$
as a polynomial in the indeterminate $u$, where $s_\lambda(h)$ is the Schur polynomial evaluated at the eigenvalues $h_1, \dots, h_N$ of $h$. We emphasize that the definition of $T_\lambda(u)$ depends on the choice of $N$, $n$, $h_1, \dots, h_N\in\C$, and $z_1, \dots, z_n\in\C$.

We will also consider the specialization of $T_\lambda(u)$ at $u=t$ (where $t\in\C$), denoted $T_\lambda(t)$. Observe that $T_\lambda(t)$ is obtained from $T_\lambda := T_\lambda(0)$ by replacing each $z_k$ with $z_k+t$. We have
\begin{align}\label{defn_Tlambda_expanded}
T_\lambda = \sum_{K\subseteq [n]}z_{[n]\setminus K}\cdot\dd_Ks_\lambda(h).
\end{align}

Note that $T_\lambda(u) = 0$ if $\ell(\lambda) > N$. Also, $T_\lambda(u)$ is invariant under conjugating $h$, and hence only depends on the eigenvalues of $h$. Therefore we may assume that $h = \Diag(h_1, \dots, h_N)$ is a diagonal matrix.

\begin{example}\label{eg_Tlambda}
Let us calculate $T_\lambda$ for all partitions $\lambda$ of size at most $2$ (cf.\ \cite[(3.12)]{alexandrov_leurent_tsuboi_zabrodin14}); then $T_\lambda(u)$ is obtained from $T_\lambda$ by replacing each $z_k$ by $z_k + u$. By \eqref{schur_expansion} and \eqref{power_sum_trace}, we have
$$
s_\varnothing(h) = 1, \quad s_{(1)}(h) = \Tr(h), \quad s_{(2)}(h) = \frac{\Tr(h)^2 + \Tr(h^2)}{2}, \quad s_{(1,1)}(h) = \frac{\Tr(h)^2 - \Tr(h^2)}{2}.
$$
Since $\dd(1) = 0$, we get
$$
T_\varnothing = z_{[n]} = z_1 \cdots z_n.
$$
Now we calculate $\dd(\Tr(h)^k) \in \End(\C^N)$:
$$
\dd(\Tr(h)^k) = \sum_{1 \le i,j \le N}\partial_{\varepsilon}\eval{\varepsilon=0}(\Tr(h + \varepsilon E_{j,i}))^k \cdot E_{i,j} = \sum_{i=1}^Nk\hspace*{1pt}\Tr(h)^{k-1} \cdot E_{i,i} = k\hspace*{1pt}\Tr(h)^{k-1}\cdot I,
$$
where $I$ denotes the identity operator. Hence
$$
T_{(1)} = z_{[n]}\Tr(h) + \sum_{k=1}^n z_{[n]\setminus\{k\}}.
$$

Next, we calculate $\dd\big(\frac{\Tr(h^2)}{2}\big) \in \End(\C^N)$:
$$
\dd\bigg(\frac{\Tr(h^2)}{2}\bigg) = \sum_{1 \le i,j \le N}\partial_{\varepsilon}\eval{\varepsilon=0}\frac{\Tr((h + \varepsilon E_{j,i})^2)}{2} \cdot E_{i,j} = \sum_{1 \le i,j \le N}h_{i,j}\cdot E_{i,j} = h \in \End(\C^N).
$$
Finally, we calculate $\dd_2\dd_1\big(\frac{\Tr(h^2)}{2}\big) \in \End(\C^N \otimes \C^N)$:
$$
\dd_2\dd_1\bigg(\frac{\Tr(h^2)}{2}\bigg) = \dd_2(h^{(1)}) = \sum_{1 \le i,j \le N}\partial_{\varepsilon}\eval{\varepsilon=0}(h + \varepsilon E_{j,i})^{(1)} \cdot E_{i,j}^{(2)} = \sum_{1 \le i,j \le N}E_{j,i}^{(1)}E_{i,j}^{(2)}  = \sigma_{1,2},
$$
where $\sigma_{k,l}\in\SG{n}$ (for distinct $k,l\in [n]$) denotes the transposition swapping $k$ and $l$. Hence
\begin{align*}
T_{(2)} &= \frac{z_{[n]}}{2}(\Tr(h)^2 + \Tr(h^2)) + \sum_{k=1}^nz_{[n]\setminus\{k\}}(\Tr(h) + h^{(k)}) + \sum_{1 \le k < l \le n}z_{[n]\setminus\{k,l\}}(I + \sigma_{k,l}), \\
T_{(1,1)} &= \frac{z_{[n]}}{2}(\Tr(h)^2 - \Tr(h^2)) + \sum_{k=1}^nz_{[n]\setminus\{k\}}(\Tr(h) - h^{(k)}) + \sum_{1 \le k < l \le n}z_{[n]\setminus\{k,l\}}(I - \sigma_{k,l}),
\end{align*}
completing the calculation.
\end{example}

The following key properties of $T_\lambda(u)$ were established in \cite{alexandrov_leurent_tsuboi_zabrodin14}:
\begin{thm}{\cite[Section 4.1 and (4.29)]{alexandrov_leurent_tsuboi_zabrodin14}}\label{ALTZ_properties}~
\begin{enumerate}[label=(\roman*), leftmargin=*, itemsep=2pt]
\item\label{ALTZ_properties_commute} The operators $T_\lambda(u)$ pairwise commute:
\begin{align}\label{Tlambda_commute}
T_\lambda(u)T_\mu(v) = T_\mu(v)T_\lambda(u) \quad \text{ for all partitions $\lambda,\mu$}.
\end{align}
\item\label{ALTZ_properties_determinant} For all partitions $\lambda$, we have
\begin{align}\label{equation_ALTZ_determinant}
\frac{T_\lambda(u)}{(u+z_1) \cdots (u+z_n)} = \det\left(\sum_{k=0}^{j-1}(-1)^k\binom{j-1}{k}\du{u}{k}\frac{T_{(1^{\conjugate{\lambda}_i-i+j-k})}(u)}{(u+z_1) \cdots (u+z_n)}\right)_{1 \le i,j \le \lambda_1}
\end{align}
as rational functions of $u$, where $\conjugate{\lambda}$ denotes the conjugate of $\lambda$ and $(1^m)$ denotes a single-column partition.\hfill\qed
\end{enumerate}
\end{thm}

We mention that the commutation relations \eqref{Tlambda_commute} are a straightforward consequence of the Yang--Baxter equation, since the $T_\lambda(u)$'s are appropriate limits of Yangian transfer matrices.

\subsection{Universal differential operator}\label{sec_dop}
Given an $N$-dimensional vector space $V\in\Gr(N,\entire)$ of entire functions, the monic linear differential operator $\DD_V$ of order $N$ with kernel $V$ applied to a function $g$ is given by
\begin{align}\label{differential_operator}
\DD_V(g) = \frac{\Wr(f_1, \dots, f_N, g)}{\Wr(f_1, \dots, f_N)} = \du{u}{N}g + \cdots.
\end{align}
We have the following formula for $\DD_V$:
\begin{lem}\label{D_kernel}
Let $V\in\Gr(N,\entire)$ be a space of entire functions. Then
$$
\DD_V = \sum_{i=0}^N(-1)^i\hspace*{1pt}\frac{\Delta_{(1^i)}(V(u))}{\Delta_\varnothing(V(u))}\hspace*{1pt}\du{u}{N-i}.
$$
\end{lem}

\begin{proof}
This follows from \eqref{differential_operator} by writing $\Wr(f_1, \dots, f_N, g)$ as in \eqref{wronskian_formula}, expanding the determinant along the last column, and applying \eqref{plucker_translation}.
\end{proof}

For spaces $V$ of quasi-exponentials, a universal formula for $\DD_V$ was obtained in \cite{mukhin_tarasov_varchenko06,mukhin_tarasov_varchenko08}. Namely, we define the \defn{universal differential operator}
\begin{align}\label{defn_universal}
\dop^\betheh := \du{u}{N} + \sum_{i=1}^NB_i(u)\du{u}{N-i}, \quad \text{ where } B_i(u) = \sum_{j=0}^\infty B_{i,j}u^{-j}
\end{align}
is a certain formal power series in $u^{-1}$ with coefficients $B_{i,j} \in U(\gl_N[t])$ depending on $h_1, \dots, h_N\in\C$; see \cite[Section 3.1]{mukhin_tarasov_varchenko08} for the explicit definition. The coefficients $B_{i,j}$ pairwise commute and generate the \defn{$\gl_N$-Bethe algebra (for the Gaudin model)} $\betheh\subseteq U(\gl_N[t])$. 

For each $z\in\C$, the Bethe algebra $\betheh$ acts on the \defn{evaluation module} $\C^N(z)$, which is isomorphic to $\C^N$ as a vector space with $t$ acting as multiplication by $z$. Let $\betheh(z_1, \dots, z_n)$ denote the image of $\betheh$ in $\End(\C^N(-z_1) \otimes \cdots \otimes \C^N(-z_n))$. Then $\dop^\betheh$ is universal for spaces of quasi-exponentials in the following sense:
\begin{thm}{\cite[Theorem 7.1]{mukhin_tarasov_varchenko08}}\label{MTV_operator}
Let $h_1, \dots, h_N\in \C$ and $z_1, \dots, z_n\in \C$.
\begin{enumerate}[label=(\roman*), leftmargin=*, itemsep=2pt]
\item\label{MTV_operator_forward} If $E\subseteq (\C^N)^{\otimes n}$ is an eigenspace of $\betheh(z_1, \dots, z_n)$, then the kernel of the differential operator obtained from $\dop^\betheh$ by replacing each $B_{i,j}$ by its eigenvalue is a space of quasi-exponentials $V_E\in\Gr(N,\entire)$ with exponents $h_1, \dots, h_N$ and Wronskian $e^{(h_1 + \cdots + h_N)u}(u+z_1)\cdots (u+z_n)$.
\item\label{MTV_operator_backward} Every space of quasi-exponentials $V\in\Gr(N,\entire)$ with exponents $h_1, \dots, h_N$ and Wronskian $e^{(h_1 + \cdots + h_N)u}(u+z_1)\cdots (u+z_n)$ equals $V_E$ for some eigenspace $E$, as in part \ref{MTV_operator_forward}.\hfill\qed
\end{enumerate}
\end{thm}

We mention that the proof of the reality theorem for spaces of quasi-exponentials (\cref{MTV_quasiexp}) in \cite[Section 7.1]{mukhin_tarasov_varchenko08} uses \cref{MTV_operator}, along with the fact that if $h_1, \dots, h_N$ and $z_1, \dots, z_n$ are real, then each $B_{i,j}$ is a self-adjoint operator, and hence has real eigenvalues.

The following result is our link between $\dop^\betheh$ and the $T_\lambda(u)$'s:
\begin{lem}\label{equality_columns}
For all $z_1, \dots, z_n\in\C$ and $1 \le i \le N$, we have
\begin{align}\label{equation_equality_columns}
B_i(u) = \frac{(-1)^i\,T_{(1^i)}(u)}{(u+z_1)\cdots(u+z_n)} \quad \text{ in } \End(\C^N(-z_1) \otimes \cdots \otimes \C^N(-z_n))
\end{align}
as rational functions of $u$.
\end{lem}

\begin{proof}
Briefly, this follows from the fact that both the universal differential operator $\dop^\betheh$ in \cite{mukhin_tarasov_varchenko06} and the operators $T_{(1^i)}(u)$ in \cite{alexandrov_leurent_tsuboi_zabrodin14} are obtained from a construction of Talalaev \cite{talalaev06}. Explicitly, following \cite[(8.4)]{mukhin_tarasov_varchenko06} we can write
$$
\dop^\betheh = \du{u}{N} + \sum_{i=1}^N(-1)^i\mathcal{G}_{i,h}(u)\du{u}{N-i}, \quad \text{ where } \mathcal{G}_{i,h}(u) = (-1)^iB_i(u).
$$
Now for $g\in\GL_N$ and $0 \le i \le N$, let $\mathcal{T}_{i,g}(u)\in Y(\gl_N)\otimes \C[[u^{-1}]]$ denote the Yangian transfer matrix as in \cite[(4.13)]{mukhin_tarasov_varchenko06}. By \cite[(10.3) and (10.5)]{mukhin_tarasov_varchenko06}, for all $1 \le i \le N$ we have
\begin{align}\label{MTV_S}
\mathcal{G}_{i,h}(u) = \lim_{\eta\to 0}(\eta^{-i}\mathcal{S}_{i,e^{\eta h}}), \quad \text{ where } \mathcal{S}_{i,g}(u) = \sum_{j=0}^i(-1)^{i-j}\binom{N-j}{N-i}\mathcal{T}_{j,g}(u).
\end{align}

On the other hand, in the notation of \cite[(2.3)]{alexandrov_leurent_tsuboi_zabrodin14}, the action of $\mathcal{T}_{i,g}(u)$ on $\C^N(-z_1) \otimes \cdots \otimes \C^N(-z_n)$ is denoted by $\mathsf{T}_\lambda(u)$ (with $x_k := -z_k$), where $\lambda = (1^i)$. By \cite[Section 3]{alexandrov_leurent_tsuboi_zabrodin14}, for all $1 \le i \le N$ we have
\begin{gather}\label{ALTZ_T}
\begin{multlined}
T_{(1^i)}(u) = (u+z_1) \cdots (u+z_n)\lim_{\eta\to 0}(\eta^{-i}\widetilde{\mathsf{T}}_{(1^i)}(u)\eval{g = e^{\eta h}}), \\[2pt] \hspace*{84pt} \text{ where } \widetilde{\mathsf{T}}_{(1^i)}(u) = \sum_{j=0}^i(-1)^{i-j}\binom{N-j}{N-i}\mathsf{T}_{(1^j)}(u).
\end{multlined}
\end{gather}
Comparing \eqref{MTV_S} and \eqref{ALTZ_T}, we obtain \eqref{equation_equality_columns}.
\end{proof}

\begin{cor}\label{Tlambda_generates}
We have $T_\lambda(t) \in \betheh(z_1, \dots, z_n)$ for all $t\in \C$ and partitions $\lambda$.
\end{cor}

\begin{proof}
This follows from \eqref{equation_ALTZ_determinant} and \eqref{equation_equality_columns}.
\end{proof}

We will also need the following property of $\betheh(z_1, \dots, z_n)$:
\begin{thm}{\cite[Theorem 9.1]{mukhin_tarasov_varchenko06}}\label{bethe_properties}~
For generic $z_1, \dots, z_n\in\C$, the algebra $\betheh(z_1, \dots, z_n)$ is semisimple (i.e.\ diagonalizable).\hfill\qed
\end{thm}

\section{Jacobi--Trudi identities for Pl\"{u}cker coordinates}\label{sec_JT}

\noindent In this section, we prove two identities for the Pl\"{u}cker coordinates of $V\in\Gr(N,\entire)$ (see \cref{JT}). We call them \defn{Jacobi--Trudi identities} because in the case that $V$ is a space of pure exponentials, they recover the Jacobi--Trudi identities for Schur polynomials (see \cref{eg_JT}).

We first state \cref{JT}, and then devote the remainder of the section to its proof. Recall that $(i)$ denotes a single-row partition, and $(1^i)$ denotes a single-column partition. By convention, we set $\Delta_{(i)}(V),\hspace*{1pt} \Delta_{(1^i)}(V) := 0$ for $i < 0$.
\begin{thm}\label{JT}
Let $V\in\Gr(N,\entire)$ be a space of entire functions, and fix a scaling $(\Delta_\lambda(V(u)))_\lambda$ of the vector of Pl\"{u}cker coordinates of $V(u)$ so that every $\Delta_\lambda(V(u))$ is a meromorphic function, such as via \eqref{plucker_translation}.
\begin{enumerate}[label=(\roman*), leftmargin=*, itemsep=2pt]
\item\label{JT_1} (Jacobi--Trudi identity) For all partitions $\lambda$ and $m \ge \ell(\lambda)$, we have
\begin{align}\label{equation_JT_1}
\Delta_\lambda(V(u)) = \Delta_\varnothing(V(u))^{1-m}\cdot\det\hspace*{-1pt}\Bigg(\sum_{k=0}^{j-1}(-1)^k\binom{j-1}{k}\du{u}{k}\Delta_{(\lambda_i - i + j - k)}(V(u))\Bigg)_{\hspace*{-2pt}1 \le i,j \le m}.
\end{align}
\item\label{JT_2} (dual Jacobi--Trudi identity) For all partitions $\lambda$ and $m \ge \lambda_1$, we have
\begin{align}\label{equation_JT_2}
\Delta_\lambda(V(u)) = \Delta_\varnothing(V(u))^{1-m}\cdot\det\hspace*{-1pt}\Bigg(\sum_{k=0}^{j-1}(-1)^k\binom{j-1}{k}\du{u}{k}\Delta_{(1^{\conjugate{\lambda}_i - i + j - k})}(V(u))\Bigg)_{\hspace*{-2pt}1 \le i,j \le m}.
\end{align}
\end{enumerate}
\end{thm}
The identity \eqref{equation_JT_1} for $m = \ell(\lambda)$ was obtained by Natanzon and Zabrodin \cite{natanzon_zabrodin16} in the more general setting of the infinite-dimensional Sato Grassmannian. Closely related identities for $T_\lambda(u)$ appear in \cite[(4.28) and (4.29)]{alexandrov_leurent_tsuboi_zabrodin14}.

\begin{example}\label{eg_JT}
Let $V := \langle e^{h_1u}, \dots, e^{h_Nu}\rangle$ be the space of pure exponentials from \cref{eg_positivity}. Then we may fix a scaling of the vector of Pl\"{u}cker coordinates of $V(u)$ so that
\begin{align}\label{equation_schur_plucker}
\Delta_\lambda(V(u)) = s_\lambda(h_1, \dots, h_N) \quad \text{ for all partitions $\lambda$ (independent of $u$)}.
\end{align}
Then \eqref{equation_JT_1} and \eqref{equation_JT_2} become, respectively,
\begin{align*}
&s_\lambda(h_1, \dots, h_N) = \det\hspace*{-1pt}\big(s_{(\lambda_i - i + j)}(h_1, \dots, h_N)\big)_{1 \le i,j \le m} && \hspace*{-24pt}\text{for all } m \ge \ell(\lambda); \text{ and} \\
&s_\lambda(h_1, \dots, h_N) = \det\hspace*{-1pt}\big(s_{(1^{\conjugate{\lambda}_i - i + j})}(h_1, \dots, h_N)\big)_{1 \le i,j \le m} && \hspace*{-24pt}\text{for all } m \ge \lambda_1.
\end{align*}
These are the well-known Jacobi--Trudi identities expressing Schur polynomials in terms of homogeneous symmetric polynomials and elementary symmetric polynomials, respectively \cite[Theorem 7.16]{stanley24}.
\end{example}

We outline the proof of \cref{JT}. In \cref{sec_basis}, we construct an explicit basis of $V\in\Gr(N,\entire)$ in terms of single-row Pl\"{u}cker coordinates $\Delta_{(i)}(V(u))$ (see \cref{h_basis}). This result can be proved using arguments in \cite[Section 5.2.3]{karp_purbhoo}, where such a basis was given for spaces $V$ of polynomials; we present a different and shorter proof. We then use this basis in \cref{sec_JT_1} to prove the first Jacobi--Trudi identity \eqref{equation_JT_1}. Finally, in \cref{sec_duality} we use Grassmann duality to prove the dual Jacobi--Trudi identity \eqref{equation_JT_2}.

\begin{rem}\label{remark_entire}
In this section, we state and prove our results for spaces $V\in\Gr(N,\entire)$ of entire functions, since that is convenient and sufficient for our purposes. However, our formulas are algebraic and hence hold more generally for spaces $V\in\Gr(N,\C[[u]])$ of formal power series, where $V(t)$ is regarded as a space of formal power series in $u$ over $\C[[t]]$. Similarly, results which depend only locally on $V(u)$ near $u=t$ (where $t\in\C$) hold for all spaces $V$ of functions which are holomorphic or real analytic at $u=t$.
\end{rem}

\subsection{A basis from single-row Pl\"{u}cker coordinates}\label{sec_basis}
We construct an explicit basis of any space $V\in\Gr(N,\entire)$ of entire functions. Fix a scaling $(\Delta_\lambda(V(u)))_\lambda$ of the vector of Pl\"{u}cker coordinates of $V(u)$ so that every $\Delta_\lambda(V(u))$ is an entire function, such as via \eqref{plucker_translation}. For $j\ge 0$, we define
$$
\hh_V(t,u) := \sum_{i\ge 0}\frac{\Delta_{(i)}(V(t))}{(N+i-1)!}(u-t)^{N+i-1} \quad \text{ and } \quad \hh^{(j)}_V(t,u) := \du{t}{j}\hh_V(t,u).
$$
(Our convention in defining $\hh_V(t,u)$ differs from \cite{karp_purbhoo}, in that we have negated $t$.) The fact that $\hh_V(t,u)$ and $\hh^{(j)}_V(t,u)$ are well-defined analytic functions of $(t,u)$ will follow from \eqref{h_determinant_formula} below.
\begin{example}\label{eg_h}
Let $V := \langle e^{h_1u}, \dots, e^{h_Nu}\rangle$ be the space of pure exponentials from \cref{eg_positivity,eg_JT}, and fix the scaling \eqref{equation_schur_plucker} of the vector of Pl\"{u}cker coordinates of $V(u)$. Then
\begin{gather*}
\hh_V(t,u) = \sum_{i\ge 0}\frac{s_{(i)}(h_1, \dots, h_N)}{(N+i-1)!}(u-t)^{N+i-1}.\qedhere
\end{gather*}
\end{example}

\begin{rem}\label{change_of_scaling}
Note that given a choice of scaling $(\Delta_\lambda(V(u)))_\lambda$ so that every $\Delta_\lambda(V(u))$ is an entire function, every other scaling can be written as $(\phi(u)\cdot\Delta_\lambda(V(u)))_\lambda$, where $\phi(u)$ is a meromorphic function. This change of scaling replaces $\hh_V(t,u)$ by $\phi(t)\cdot\hh_V(t,u)$, and hence replaces $\hh^{(j)}_V(t,u)$ by $\sum_{k=0}^j\binom{j}{k}\du{t}{j-k}\phi(t)\cdot\hh^{(k)}_V(t,u)$.
\end{rem}

The main result of this subsection is as follows:
\begin{thm}\label{h_basis}
Let $V\in\Gr(N,\entire)$ be a space of entire functions in the variable $u$. Then
\begin{align}\label{h_span_equation}
V = \langle \hh_V(t,u) : t\in\C\rangle.
\end{align}
Moreover, for every $t\in\C$ which is not a zero of $\Wr(V)$, the following is a basis of $V$:
\begin{align}\label{h_basis_equation}
(\hh_V(t,u), \hh^{(1)}_V(t,u), \dots, \hh^{(N-1)}_V(t,u)).
\end{align}
\end{thm}

\begin{proof}
By \cref{change_of_scaling}, it suffices to prove the result for a single choice of scaling $(\Delta_\lambda(V(u)))_\lambda$. Let $(f_1, \dots, f_N)$ be a basis of $V$, and fix the scaling \eqref{plucker_translation} of the vector of Pl\"{u}cker coordinates of $V(t)$ for all $t\in\C$. We claim that
\begin{align}\label{h_determinant_formula}
\hh_V(t,u) = \det\begin{bmatrix}
f_1(t) & \cdots & f_N(t) \\[2pt]
f_1'(t) & \cdots & f_N'(t) \\
\vdots & \ddots & \vdots \\[2pt]
f_1^{(N-2)}(t) & \cdots & f_N^{(N-2)}(t) \\[2pt]
f_1(u) & \cdots & f_N(u)
\end{bmatrix}.
\end{align}
To see this, we write each $f_j(u)$ in the last row as a Taylor series about $u=t$, and then expand the determinant along the last row.

Notice that \eqref{h_determinant_formula} implies that $\hh^{(j)}_V(t,u)$ is a well-defined analytic function of $(t,u)$. Also, expanding the determinant in \eqref{h_determinant_formula} along the last row shows that $\hh_V(t,u)\in V$. We then get that $\hh^{(j)}(t,u) \in V$ as a consequence of the general property of derivatives
\begin{align}\label{h_derivative}
\hh^{(j)}_V(t,u) \in \langle \hh_V(s,u) : s\in\C\rangle \quad \text{ for all } j\ge 0.
\end{align}

Therefore \eqref{h_basis_equation} is contained in $V$. To show that it is a basis of $V$ (assuming $t\in\C$ is not a zero of $\Wr(V)$), it suffices to show that it is linearly independent. To see this, observe that
$$
\hh^{(j)}_V(t,u) = \frac{(-1)^j\Delta_\varnothing(V(t))}{(N-j-1)!}(u-t)^{N-j-1} + \text{higher-degree terms in $(u-t)$}
$$
for all $0 \le j \le N-1$, and $\Delta_\varnothing(V(t)) \neq 0$ by \eqref{equation_wronskian_empty}.

Finally, we prove \eqref{h_span_equation}. We showed above that $\langle \hh_V(t,u) : t\in\C\rangle \subseteq V$. The reverse containment follows from \eqref{h_derivative}, using any basis of the form \eqref{h_basis_equation}.
\end{proof}

\begin{rem}
In the setting of \cref{h_basis}, if $t$ is a zero of $\Wr(V)$ and $N\ge 2$, then \eqref{h_basis_equation} is linearly dependent. This was proved in \cite[Remark 5.16]{karp_purbhoo} for spaces of polynomials, and a similar argument applies to spaces of entire functions.
\end{rem}

\subsection{Proof of the first Jacobi--Trudi identity}\label{sec_JT_1}
We prove \eqref{equation_JT_1}. Notice that the matrix in \eqref{equation_JT_1} in rows $\ell(\lambda)+1, \dots, m$ is upper-triangular with $\Delta_\varnothing(V(u))$ on the diagonal, so the right-hand side of \eqref{equation_JT_1} does not depend on $m \ge \ell(\lambda)$. Similarly, \eqref{equation_JT_1} holds when $\lambda = \varnothing$. Therefore it suffices to prove \eqref{equation_JT_1} up to rescaling for a single choice of $m\ge\ell(\lambda)$. One such proof is provided by \cite[Theorem 2.1]{natanzon_zabrodin16}. We give a different argument using \cref{sec_basis}.

By \cref{remark_entire}, we may assume that the scaling $(\Delta_\lambda(V(u)))_\lambda$ consists of entire functions. It suffices to prove \eqref{equation_JT_1} at $u=t$ for every $t\in\C$ which is not a zero of $\Wr(V)$. By translation, we may assume that $t=0$. First we consider the case when $\ell(\lambda)\le N$, so that we may take $m=N$. By \eqref{plucker_translation} and using the basis of $V$ from \eqref{h_basis_equation} with $t=0$, we get
$$
\Delta_\lambda(V) = \det\hspace*{-1pt}\big((-1)^{j-1}\du{u}{\lambda_i+N-i}\eval{u=0}\hh^{(j-1)}_V(0,u)\big)_{1 \le i,j \le N}.
$$
Then calculating that
$$
\du{u}{i}\eval{u=0}\hh^{(j)}_V(0,u) = \sum_{k=0}^j(-1)^{j-k}\binom{j}{k}\du{t}{k}\eval{t=0}\Delta_{(i+j+1-k-N)}(V(t)) \quad \text{ for all } i,j\ge 0,
$$
we obtain \eqref{equation_JT_1} up to rescaling when $m=N$.

Now we consider the case when $\ell(\lambda) > N$. We define
$$
W := \{f(u)\in\entire : \du{u}{\ell(\lambda) - N}f(u)\in V\} \in \Gr(\ell(\lambda), \entire),
$$
so that $\Delta_\mu(V) = \Delta_\mu(W)$ for all partitions $\mu$ (see \cite[Proposition 2.10(ii)]{karp_purbhoo}). We showed above that \eqref{equation_JT_1} holds for $W$ when $m = \ell(\lambda)$, so it also holds for $V$.\hfill\qed

\subsection{Grassmann duality and proof of the dual Jacobi--Trudi identity}\label{sec_duality}
We recall Grassmann duality for spaces of polynomials, following \cite[Section 5.3.1]{karp_purbhoo}, which we then use to prove the dual Jacobi--Trudi identity \eqref{equation_JT_2}. Let $\C[u]_{\le M-1}\cong\C^M$ denote the vector space of all polynomials of degree at most $M-1$. Define the nondegenerate bilinear pairing $(\cdot,\cdot)$ on $\C[u]_{\le M-1}$ by
$$
\bigg(\sum_{i=0}^{M-1}\frac{f_i}{i!}u^i, \sum_{j=0}^{M-1}\frac{g_j}{j!}u^j\bigg) := \sum_{i=0}^{M-1}(-1)^if_ig_{M-1-i}.
$$
Given $V\in\Gr(N,\C[u]_{\le M-1})$, define its \defn{dual}
$$
\conjugate{V} := \{f\in\C[u]_{\le M-1} : (f,g) = 0 \text{ for all } g\in V\} \in \Gr(M-N,\C[u]_{\le M-1}).
$$
Recall that $\conjugate{\lambda}$ denotes the conjugate partition of $\lambda$. We will need the following properties:
\begin{lem}{\cite[Proposition 5.17]{karp_purbhoo}}\label{duality}
Let $V\in\Gr(N,\C[u]_{\le M-1})$.
\begin{enumerate}[label=(\roman*), leftmargin=*, itemsep=2pt]
\item\label{duality_pluckers} Taking duals preserves  Pl\"{u}cker coordinates: $\Delta^\lambda(V) = \Delta^{\conjugate{\lambda}}(\conjugate{V})$ for all partitions $\lambda$.
\item\label{duality_translation} Taking duals commutes with translation: $\conjugate{V(t)} = \conjugate{V}(t)$ for all $t\in\C$.\hfill\qed
\end{enumerate}
\end{lem}

\begin{proof}[Proof of \eqref{equation_JT_2}]
By translation, it suffices to prove \eqref{equation_JT_2} at $u=0$. Given a power series $f(u)$ about $u=0$, let $\pi_k(f(u))\in\C[u]_{\le k}$ denote the polynomial obtained by truncating the power series at degree $k$. We expand both sides of \eqref{equation_JT_2} as power series about $u=0$; it then suffices to show that equality holds after applying $\pi_k$ to both sides, for all $k\ge 0$. We claim that for every fixed $k$, the result of applying $\pi_k$ to each side only depends on $\pi_c(V)$ for some $c$ (depending on $k$). This implies that it suffices to prove \eqref{equation_JT_2} when $V$ is a space of polynomials, whence it follows from \eqref{equation_JT_1} by duality, using \cref{duality}.

To justify the claim, fix a basis $(f_1, \dots, f_N)$ of $V$. By \eqref{plucker_translation}, we can write the given scaling of the vector of Pl\"{u}cker coordinates of $V(u)$ as
$$
\Delta_\lambda(V(u)) = \phi(u)\cdot\det\hspace*{-1pt}\big(f_j^{(\lambda_{N-i+1}+i-1)}(u))_{1 \le i,j \le N}
$$
for all partitions $\lambda$ with $\ell(\lambda) \le N$, where $\phi(u)$ is some meromorphic function. We see that $\pi_l(\Delta_\lambda(V(u)))$ only depends on $\pi_{l+\lambda_1+N-1}(V)$, and the claim follows.
\end{proof}

\section{Universal Pl\"{u}cker coordinates}\label{sec_universal}

\noindent In this section, we prove that the higher Gaudin Hamiltonians $T_\lambda(t)$ are universal Pl\"{u}cker coordinates for spaces of quasi-exponentials about $u=t$:
\begin{thm}\label{quasiexp_coordinates}
Let $h_1, \dots, h_N\in \C$ and $z_1, \dots, z_n\in \C$, and fix $t\in\C$. Let $E\subseteq(\C^N)^{\otimes n}$ be an eigenspace of $\betheh(z_1, \dots, z_n)$, and let $V_E\in\Gr(N,\entire)$ be the space of quasi-exponentials from \cref{MTV_operator}\ref{MTV_operator_forward}. For every partition $\lambda$, let $T_{\lambda,E}(t)$ be the eigenvalue of $T_\lambda(t)$ acting on $E$. Then $(T_{\lambda,E}(t))_{\lambda}$ is the vector of Pl\"{u}cker coordinates of $V_E$ about $u=t$, i.e.,
$$
(T_{\lambda,E}(t))_{\lambda} = (\Delta_\lambda(V_E(t)))_\lambda \quad \text{ up to rescaling}.
$$
\end{thm}

\begin{proof}
By \cref{Tlambda_generates}, $E$ is indeed a common eigenspace of the operators $T_\lambda(t)$.  It suffices to show that $(T_{\lambda,E}(t))_\lambda = (\Delta_\lambda(V_E(t)))_\lambda$ up to rescaling when $t\in\C$ is generic, in which case it follows from \cref{D_kernel}, \eqref{equation_JT_2}, \eqref{equation_equality_columns}, and \eqref{equation_ALTZ_determinant}.
\end{proof}

The following result implies that the eigenspaces $E$ appearing in \cref{quasiexp_coordinates} are precisely the common eigenspaces of the operators $T_\lambda(t)$, for any fixed $t\in\C$:
\begin{prop}\label{Tlambda_generates_strong}
Fix $t\in\C$. For all $z_1, \dots, z_n\in\C$, the algebra $\betheh(z_1, \dots, z_n)$ is generated by $\{T_\lambda(t) : \lambda \text{ is a partition}\}$.
\end{prop}

\begin{proof}
By translation, we may assume that $t=0$. By \cref{Tlambda_generates}, it suffices to show that each $B_i(u)$ is generated by the $T_\lambda$'s. By \cref{bethe_properties}, \cref{MTV_operator}, and \cref{D_kernel}, we can treat $B_i(u)$ (and hence also $T_{(1^i)}(u)$, by \eqref{equation_equality_columns}) as a Pl\"{u}cker coordinate $\Delta_{(1^i)}(V(u))$. Then by \eqref{equation_ALTZ_determinant} and \eqref{equation_JT_2}, we can treat $T_\mu(u)$ as a Pl\"{u}cker coordinate $\Delta_\mu(V(u))$. In particular, $T_\mu(u)$ is expressed in terms of the $T_\lambda$'s precisely as in \eqref{translation_identity}, up to a scalar (depending on $u$). Taking $\mu = (1^i)$ shows that $B_i(u)$ is generated by the $T_\lambda$'s.
\end{proof}

\section{Partial-trace formula}\label{sec_formula_proof}

\noindent In this section, we present a formula for the higher Gaudin Hamiltonians $T_\lambda(u)$ in terms of partial traces (see \cref{Tlambda_formula}). In order to prove it, we need several preliminary identities. Given a finite set $K$ and partition $\lambda\vdash |K|$, we define
\begin{align}\label{defn_alpha}
\alpha_\lambda^{(K)} := \sum_{\sigma\in\SG{K}}\chi^\lambda(\sigma)\sigma \in\C[\SG{K}].
\end{align}
Since the character $\chi^\lambda$ is constant on conjugacy classes, $\alpha_\lambda^{(K)}$ is in the center $Z(\C[\SG{K}])$.
\begin{lem}\label{identities}
Let $\lambda\vdash r$, let $K \subseteq L$ with $|L| = r$, let $\sigma\in\SG{L}$, and let $\gamma\in Z(\C[\SG{L}])$. Then
\begin{align}
&\label{identities_power} \Tr_L(h^{(L)}\sigma) = p_{\cyc(\sigma)}(h); \\[4pt]
&\label{identities_schur} \Tr_L(h^{(L)}\alpha_\lambda^{(L)}) = r!s_\lambda(h); \\[4pt]
&\label{identities_derivative} \dd_K\Tr_L(h^{(L)}\gamma) = (r)_{|K|}\Tr_{L\setminus K}(h^{(L\setminus K)}\gamma); \\[4pt]
&\label{identities_commute} \Tr_{L\setminus K}(h^{(L\setminus K)}\sigma) = \Tr_{L\setminus K}(\sigma h^{(L\setminus K)}); \\[4pt]
&\label{identities_factor} \Tr_{L\setminus K}(h^{(L)}\sigma) = h^{(K)}\Tr_{L\setminus K}(h^{(L\setminus K)}\sigma), \quad \Tr_{L\setminus K}(\sigma h^{(L)}) = \Tr_{L\setminus K}(h^{(L\setminus K)}\sigma)h^{(K)}.
\end{align}
The last three formulas are identities in $\operatorname{End}(\C^N)^{\otimes K}$.
\end{lem}

\begin{proof}
We may assume that $K = [k]$, $L = [r]$, and $h = \Diag(h_1, \dots, h_N)$ is diagonal. We will use the following identity several times:
\begin{align}\label{useful_identity}
h^{(S)}\sigma = \sum_{1 \le i_1, \dots, i_r \le N}\Big(\prod_{l \in S} h_{i_{\sigma^{-1}(l)}}\Big)\Big(\prod_{l=1}^r E^{(l)}_{i_{\sigma^{-1}(l)}, i_l}\Big) \quad \text{ for all } S\subseteq [r].
\end{align}
It can be verified by evaluating each side on the elementary tensor $\evec{i_1} \otimes \cdots \otimes \evec{i_r}$, for all $1 \le i_1, \dots, i_r \le N$. We now individually establish the five identities \eqref{identities_power}--\eqref{identities_factor}:

\eqref{identities_power}:
It suffices to prove the result when $\sigma$ is an $r$-cycle. By \eqref{useful_identity} with $S=[r]$, we get
$$
\Tr(h^{([r])}\sigma) = \sum_{1 \le i_1, \dots, i_r \le N}(h_{i_1}\cdots h_{i_r})\hspace*{2pt}\Big(\prod_{l=1}^r \Tr(E_{i_{\sigma^{-1}(l)}, i_l})\Big) = \sum_{i=1}^Nh_i^r = p_r(h),
$$
as desired.

\eqref{identities_schur}:
This follows from \eqref{identities_power} and \eqref{schur_expansion}.

\eqref{identities_derivative}:
The left-hand side $\dd_{[k]}\Tr(h^{([r])}\gamma)$ equals
\begin{align*}
& \scalebox{0.94}{$\displaystyle\sum_{1 \le a_1, b_1, \dots, a_k, b_k \le N}\partial_{\varepsilon_k} \cdots \partial_{\varepsilon_1}\eval{\varepsilon_1 = \cdots = \varepsilon_k =0}\Tr\big((h + \varepsilon_1E_{b_1,a_1} + \cdots + \varepsilon_k E_{b_k,a_k}\big)^{([r])}\gamma) \cdot E^{(1)}_{a_1,b_1}\cdots E^{(k)}_{a_k,b_k}$} \\[4pt]
=& \scalebox{0.94}{$\displaystyle\sum_{1 \le a_1, b_1, \dots, a_k, b_k \le N}\hspace*{4pt}\sum_{\pi : [k] \hookrightarrow [r] \text{ injective}}\Tr(h^{([r]\setminus\pi([k]))} E^{(\pi(1))}_{b_1,a_1} \cdots E^{(\pi(k))}_{b_k,a_k}\gamma) \cdot E^{(1)}_{a_1,b_1}\cdots E^{(k)}_{a_k,b_k}$}.
\end{align*}
Since $\gamma\in Z(\C[\SG{r}])$, the last summand above does not depend on $\pi$. Hence we may take $\pi$ to be the identity on $[k]$, and multiply by the number of injections $[k]\hookrightarrow [r]$:
$$
\dd_{[k]}\Tr(h^{([r])}\gamma) = (r)_k\cdot\sum_{1 \le a_1, b_1, \dots, a_k, b_k \le N}\Tr(h^{([r]\setminus [k])} E^{(1)}_{b_1,a_1} \cdots E^{(k)}_{b_k,a_k}\gamma) \cdot E^{(1)}_{a_1,b_1}\cdots E^{(k)}_{a_k,b_k}.
$$
Therefore it suffices to prove that for all $\sigma\in\SG{r}$, we have
\begin{align}\label{new_identities_derivative}
\sum_{1 \le a_1, b_1, \dots, a_k, b_k \le N}\Tr(h^{([r]\setminus [k])} E^{(1)}_{b_1,a_1} \cdots E^{(k)}_{b_k,a_k}\sigma) \cdot E^{(1)}_{a_1,b_1}\cdots E^{(k)}_{a_k,b_k} = \Tr_{[r]\setminus [k]}(h^{([r]\setminus [k])}\sigma).
\end{align}
By \eqref{useful_identity} with $S = [r]\setminus [k]$ and \eqref{equation_partial_trace_general}, the right-hand side of \eqref{new_identities_derivative} equals
\begin{gather}\label{derivative_to_prove}
\begin{aligned}
\Tr_{[r]\setminus [k]}(h^{([r]\setminus [k])}\sigma) &= \sum_{1 \le i_1, \dots, i_r \le N}\Big(\prod_{l = k+1}^r h_{i_{\sigma^{-1}(l)}}\Big)\Big(\prod_{l=k+1}^r \Tr(E_{i_{\sigma^{-1}(l)}, i_l})\Big)\Big(\prod_{l=1}^k E^{(l)}_{i_{\sigma^{-1}(l)}, i_l}\Big) \\
&= \sum_{\substack{1 \le i_1, \dots, i_r \le N,\\[2pt] i_{\sigma^{-1}(l)} = i_l \text{ for all } l\in [r]\setminus[k]}}\Big(\prod_{l = k+1}^r h_{i_{\sigma^{-1}(l)}}\Big) \Big(\prod_{l=1}^k E^{(l)}_{i_{\sigma^{-1}(l)}, i_l}\Big).
\end{aligned}
\end{gather}
For the left-hand side of \eqref{new_identities_derivative}, we calculate that $h^{([r]\setminus [k])} E^{(1)}_{b_1,a_1} \cdots E^{(k)}_{b_k,a_k}\sigma(\evec{i_1} \otimes \cdots \otimes \evec{i_r})$ (for $1 \le i_1, \dots, i_r \le N)$ equals
$$
\Big(\prod_{l = k+1}^r h_{i_{\sigma^{-1}(l)}}\Big)\hspace*{2pt}\evec{b_1} \otimes \cdots \otimes \evec{b_k} \otimes \evec{i_{\sigma^{-1}(k+1)}} \otimes \cdots \otimes \evec{i_{\sigma^{-1}(r)}}
$$
if $a_l = i_{\sigma^{-1}(l)}$ for all $l\in [k]$, and zero otherwise. Moreover, the elementary tensor above equals $\evec{i_1} \otimes \cdots \otimes \evec{i_r}$ if and only if
$$
b_l = i_l \text{ for all } l\in [k] \quad \text{ and } \quad i_{\sigma^{-1}(l)} = i_l \text{ for all } l\in [r]\setminus[k].
$$
Hence the left-hand side of \eqref{new_identities_derivative} equals the right-hand side of \eqref{derivative_to_prove}, as desired.

\eqref{identities_commute}:
This follows from \cref{trace_commute}.

\eqref{identities_factor}:
The two equalities are symmetric; we prove the first one. By \eqref{useful_identity} with $S=[r]$, we get
$$
\Tr_{[r]\setminus [k]}(h^{([r])}\sigma) = \sum_{1 \le i_1, \dots, i_r \le N}\Big(\prod_{l = k+1}^r h_{i_{\sigma^{-1}(l)}}\Big)\Big(\prod_{l=k+1}^r \Tr(E_{i_{\sigma^{-1}(l)}, i_l})\Big)\Big(\prod_{l=1}^k h_{i_{\sigma^{-1}(l)}}E^{(l)}_{i_{\sigma^{-1}(l)}, i_l}\Big).
$$
Since $h_jE^{(l)}_{j,i} = h^{(l)}E^{(l)}_{j,i}$, we see from \eqref{derivative_to_prove} that $\Tr_{[r]\setminus [k]}(h^{([r])}\sigma) = h^{([k])}\Tr_{[r]\setminus [k]}(h^{([r]\setminus [k])}\sigma)$.
\end{proof}

We now establish the partial-trace formula for $T_\lambda(u)$:
\begin{thm}\label{Tlambda_formula}
Let $\lambda$ be a partition and $m \ge \max\{n,|\lambda|\}$. Then we have
\begin{align}\label{Tlambda_equation}
T_\lambda(u) = \sum_{\substack{K\subseteq [n],\hspace*{2pt} L\subseteq [m],\\ K\subseteq L,\hspace*{2pt} |L| = |\lambda|}}\frac{(m-|\lambda|)!}{(m-|K|)!}\,\Big(\prod_{l \in [n]\setminus K}(u+z_l)\Big)\,\Tr_{L\setminus K}(h^{(L\setminus K)}\alpha_\lambda^{(L)})
\end{align}
as elements of $\End((\C^N)^{\otimes n})$.
\end{thm}

We point out that \eqref{Tlambda_equation} involves the operator $h^{(L\setminus K)}\alpha_\lambda^{(L)}$ acting on $(\C^N)^{\otimes L}$, and $(\C^N)^{\otimes L}$ may contain some extra tensor factors not present in $(\C^N)^{\otimes n}$. However, applying the partial trace $\Tr_{L\setminus K}$ removes any dependence on the extra tensor factors, and so both sides of \eqref{Tlambda_equation} act on $(\C^N)^{\otimes n}$. Also, \eqref{Tlambda_equation} becomes cleaner if we set $m := \max\{n,|\lambda|\}$. We state \eqref{Tlambda_equation} for general $m$ so as to unify the cases $|\lambda| \le n$ and $|\lambda| \ge n$.

\begin{proof}
It suffices to prove \eqref{Tlambda_equation} at $u=0$. Let $r := |\lambda|$. By \eqref{identities_schur}, we have
\begin{align}\label{schur_average}
s_\lambda(h) = \frac{1}{r!\binom{m-|K|}{r-|K|}}\sum_{\substack{K \subseteq L \subseteq [m],\\ |L| = r}}\Tr_L(h^{(L)}\alpha_\lambda^{(L)}) \quad \text{ for all } K\subseteq [n],
\end{align}
since there are $\binom{m-|K|}{r-|K|}$ terms in the sum. Then \eqref{Tlambda_equation} follows from \eqref{defn_Tlambda_expanded}, \eqref{schur_average}, and \eqref{identities_derivative}, since $\alpha_\lambda^{(L)}\in Z(\C[\SG{L}])$ and $\frac{(r)_{|K|}}{r!\binom{m-|K|}{r-|K|}} = \frac{(m-r)!}{(m-|K|)!}$.
\end{proof}

\begin{rem}
While we did not use \eqref{identities_commute} and \eqref{identities_factor} in the proof of \cref{Tlambda_formula}, they further illuminate the formula \eqref{Tlambda_equation}. Indeed, they imply that if $h_1, \dots, h_N \neq 0$, we have
\begin{align}\label{four_ways}
\scalebox{0.88}{$(h^{-1})^{(K)}\Tr_{L\setminus K}(h^{(L)}\alpha_\lambda^{(L)}) = \Tr_{L\setminus K}(h^{(L\setminus K)}\alpha_\lambda^{(L)}) = \Tr_{L\setminus K}(\alpha_\lambda^{(L)}h^{(L\setminus K)}) = \Tr_{L\setminus K}(\alpha_\lambda^{(L)}h^{(L)})(h^{-1})^{(K)}$},
\end{align}
which gives three additional ways to rewrite \eqref{Tlambda_equation}. We will also use \eqref{four_ways} in the proof of \cref{Tlambda_positive}.
\end{rem}

\section{Positive (semi)definiteness}\label{sec_semidefinite}

\noindent In this section we prove that the higher Gaudin Hamiltonians $T_\lambda(u)$ are positive (semi)definite for appropriate choices of the parameters (see \cref{Tlambda_positive}). We then use this to prove our main result \cref{positivity} (the positivity theorem for spaces of quasi-exponentials). We will need the following three basic results:
\begin{lem}\label{semidefinite_product}
Suppose that $A,B\in\End(V)$ commute and are positive semidefinite, where $V$ is finite-dimensional. Then $AB$ is positive semidefinite.
\end{lem}

\begin{proof}
This follows by simultaneously diagonalizing $A$ and $B$.
\end{proof}

\begin{lem}\label{semidefinite_partial_trace}
Suppose that $A\in\End(V_1\otimes\cdots\otimes V_n)$ is positive semidefinite, where $V_1, \dots, V_n$ are finite-dimensional. Then the partial trace $\Tr_K(A)$ is positive semidefinite for all $K\subseteq [n]$.
\end{lem}

\begin{proof}
We may assume $n=2$ and $K = \{2\}$, whence the result follows from \eqref{partial_trace_decomposition}.
\end{proof}

Let $X$ be a finite-dimensional inner product space which is also a $\C[\SG{K}]$-module. We call $X$ \defn{unitary} if each $\sigma\in\SG{K}$ acts on $X$ as a unitary operator (cf.\ \cite[Section 1.3]{serre98}). Observe that $(\C^N)^{\otimes K}$ with the standard inner product is a unitary $\C[\SG{K}]$-module.
\begin{lem}{\cite[Theorem 8]{serre98}}\label{alpha_semidefinite}
Let $\alpha_\lambda^{(K)}\in\C[\SG{K}]$ be as in \eqref{defn_alpha}, and let $X$ be a finite-dimensional unitary $\C[\SG{K}]$-module. Then $\frac{\ff{\lambda}}{|\lambda|!}\alpha_\lambda^{(K)}$ acts on $X$ as orthogonal projection onto its $\lambda$-isotypic component. In particular, $\alpha_\lambda^{(K)}\in\End(X)$ is positive semidefinite.\hfill\qed
\end{lem}

We come to the main result of this section:
\begin{thm}\label{Tlambda_positive}
Let $\lambda$ be a partition, and suppose that $z_1, \dots, z_n$ are real.
\begin{enumerate}[label=(\roman*), leftmargin=*, itemsep=2pt]
\item\label{Tlambda_semidefinite} If $h_1, \dots, h_N \ge 0$, then $T_\lambda(t)$ is positive semidefinite for all real $t \ge -z_1, \dots, -z_n$.
\item\label{Tlambda_definite} If $h_1, \dots, h_N \ge 0$ and at least $\ell(\lambda)$ values of $h_i$ are strictly positive (in particular, if $h_1, \dots, h_N > 0$ and $\ell(\lambda) \le N$), then $T_\lambda(t)$ is positive definite for all real $t > -z_1, \dots, -z_n$.
\end{enumerate}
\end{thm}

Recall that $T_\lambda(u) = 0$ when $\ell(\lambda) > N$. Hence \cref{Tlambda_positive}\ref{Tlambda_definite} implies that if $h_1, \dots, h_N > 0$ and $t > -z_1, \dots, z_n$, then all of the nontrivial $T_\lambda(t)$'s are positive definite.

\begin{proof}
By translation, we may assume that $t=0$. Suppose that $h_1, \dots, h_N \ge 0$ and $z_1, \dots, z_n \ge 0$. We claim it suffices to show that each operator $\Tr_{L\setminus K}(h^{(L\setminus K)}\alpha_\lambda^{(L)})$ in the expansion \eqref{Tlambda_equation} of $T_\lambda$ is positive semidefinite. Indeed, then $T_\lambda$ is a sum of positive semidefinite operators with nonnegative coefficients, so it is positive semidefinite. This proves part \ref{Tlambda_semidefinite}. Moreover, the sum of terms in \eqref{Tlambda_equation} with $K=\emptyset$ is $z_{[n]} s_\lambda(h)$, by \eqref{defn_Tlambda_expanded}. If at least $\ell(\lambda)$ values of $h_i$ are positive and $z_1, \dots, z_n > 0$, then $z_{[n]} s_\lambda(h)$ is a positive scalar multiple of the identity by \eqref{schur_monomial} and hence positive definite, showing that $T_\lambda$ is positive definite. This proves part \ref{Tlambda_definite}.

To see that $\Tr_{L\setminus K}(h^{(L\setminus K)}\alpha_\lambda^{(L)})$ is positive semidefinite, we use \eqref{four_ways} to write
$$
\Tr_{L\setminus K}(h^{(L\setminus K)}\alpha_\lambda^{(L)}) = (h^{-1})^{(K)}\Tr_{L\setminus K}(h^{(L)}\alpha_\lambda^{(L)}).
$$
Note that $h^{(L)}$ and $\alpha_\lambda^{(L)}$ commute, $h^{(L)}$ is positive definite, and $\alpha_\lambda^{(L)}$ is positive semidefinite by \cref{alpha_semidefinite}, so $h^{(L)}\alpha_\lambda^{(L)}$ is positive semidefinite by \cref{semidefinite_product}. By \cref{semidefinite_partial_trace}, $\Tr_{L\setminus K}(h^{(L)}\alpha_\lambda^{(L)})$ is positive semidefinite. Finally, $(h^{-1})^{(K)}$ is positive definite, and by \eqref{four_ways} it commutes with $\Tr_{L\setminus K}(h^{(L)}\alpha_\lambda^{(L)}) = \Tr_{L\setminus K}(\alpha_\lambda^{(L)}h^{(L)})$, so $(h^{-1})^{(K)}\Tr_{L\setminus K}(h^{(L)}\alpha_\lambda^{(L)})$ is positive semidefinite by \cref{semidefinite_product}.
\end{proof}

As an immediate consequence, we obtain \cref{positivity}:
\begin{proof}[Proof of \cref{positivity}]
This follows from \cref{quasiexp_coordinates,Tlambda_positive}.
\end{proof}

\section{From polynomials to quasi-exponentials}\label{sec_poly_to_quasi}

\noindent In this section we explain why the positivity theorem for spaces of quasi-exponentials (\cref{positivity}) is in fact a consequence of the positivity theorem for spaces of polynomials (\cref{KP_positivity}). The proof proceeds in two separate steps; the first is showing that \cref{KP_positivity} implies \cref{positivity}\ref{positivity_weak}, and the second is showing that \cref{positivity}\ref{positivity_weak} implies \cref{positivity}\ref{positivity_strict}.

\subsection{Proof that \texorpdfstring{\cref{KP_positivity}}{Theorem \ref{KP_positivity}} implies \texorpdfstring{\cref{positivity}\ref{positivity_weak}}{Theorem \ref{positivity}\ref{positivity_weak}}}\label{poly_proof_weak}
We will need the following properties of the Wronskian. To prove them, we use the identity
\begin{align}\label{wronskian_induction}
\Wr(f_1, \dots, f_N) = f_1^N\cdot\Wr\hspace*{-2pt}\big(1, \textstyle\frac{f_2}{f_1}, \dots, \frac{f_N}{f_1}\big) = f_1^N\cdot\Wr\hspace*{-2pt}\big(\du{u}{}\frac{f_2}{f_1}, \dots, \du{u}{}\frac{f_N\emph{}}{f_1}\big).
\end{align}

\begin{lem}\label{wronskian_degree}
Let $V = \langle e^{h_1u}p_1(u), \dots, e^{h_Nu}p_N(u)\rangle$ be an $N$-dimensional space of quasi-exponentials. Suppose that there are $l$ distinct values among $h_1, \dots, h_N$, with multiplicities $m_1, \dots, m_l > 0$. For $1 \le i \le N$, let $d_i \ge 0$ denote the degree of $p_i$; we may perform a change of basis for $V$ so that for all $i\neq j$, if $h_i = h_j$ then $d_i \neq d_j$. Then $\Wr(V) = e^{(h_1 + \cdots + h_N)u}g(u)$, where $g(u)$ is a polynomial of degree $d_1 + \cdots + d_N - \binom{m_1}{2} - \cdots - \binom{m_l}{2}$.
\end{lem}

\begin{proof}
We proceed by induction on $N$, with base case $N=0$. For the induction step, we may assume that every $p_i$ is a monomial (i.e.\ $p_i(u) = u^{d_i}$), and we reorder the functions so that $d_1 = \min(d_1, \dots, d_N)$. The result then follows using \eqref{wronskian_induction}.
\end{proof}

\begin{lem}\label{wronskian_zero}
Let $V = \langle f_1(u), \dots, f_N(u)\rangle \in \Gr(N,\entire)$, and let $t\in\C$. For $1 \le i \le N$, let $c_i \ge 0$ denote the order of the zero of $f_i(u)$ at $u=t$ (so $c_i = 0$ if $f_i(t) \neq 0$). Suppose that there are $l$ distinct values among $c_1, \dots, c_N$, with multiplicities $m_1, \dots, m_l > 0$. Then $\Wr(V)$ has a zero at $u=t$ of order at least $c_1 + \cdots + c_N + \binom{m_1}{2} + \cdots + \binom{m_l}{2} - \binom{N}{2}$.
\end{lem}

\begin{proof}
We may perform a change of basis for $V$ so that $c_1, \dots, c_N$ become distinct. It then suffices to prove the result when $f_i(u) = (u-t)^{c_i}$ for all $1 \le i \le N$, in which case it follows from \eqref{wronskian_induction} by induction.
\end{proof}

We now prove \cref{positivity}\ref{positivity_weak}. By translation, we may assume that $t=0$, so that $z_1, \dots, z_n \ge 0$. Throughout the proof, we fix $m_1, \dots, m_l > 0$ and $d_1, \dots, d_N \ge 0$ as in \cref{wronskian_degree}, which also fixes $n = d_1 + \cdots + d_N - \binom{m_1}{2} - \cdots - \binom{m_l}{2}$.

We claim that it suffices to prove the result when $h_1, \dots, h_N \ge 0$ are integers and $z_1, \dots, z_n \ge 0$ are distinct. Indeed, this case implies the case when $h_1, \dots, h_N \ge 0$ are rationals and $z_1, \dots, z_n \ge 0$ are distinct, by rescaling the variable $u$. This then implies the case when $h_1, \dots, h_N \ge 0$ are arbitrary and $z_1, \dots, z_n \ge 0$ are distinct, by expressing each $h_i$ as a limit of rational numbers. (Here we use the fact that if $W\in\Gr(N,\entire)$ is a real space of quasi-exponentials, then $\Wr(W)$ is real, and hence its nonreal zeros come in complex-conjugate pairs. In particular, if $W$ is sufficiently close to $V$, then $\Wr(W)$ has $n$ distinct real zeros.) Finally, this implies the general case, since we can perturb the given space $V$ of quasi-exponentials so that $z_1, \dots, z_n \ge 0$ become distinct. This follows from the fact that the Wronski map $\mathcal{Q} \to \mathbb{P}(e^{(h_1 + \cdots + h_N)u}\cdot \mathbb{C}_{\le n}[u])$ is a finite morphism (see \cite[Proposition 4.2]{mukhin_tarasov_varchenko09c}), where $\mathcal{Q}$ is the variety of all spaces of quasi-exponentials of the form $\langle e^{h_1u}q_1(u), \dots, e^{h_Nu}q_N(u)\rangle$, such that $h_1, \dots, h_N$ are fixed and each $q_i$ ranges over all polynomials of degree at most $d_i$.

Thus we suppose that $h_1, \dots, h_N \ge 0$ are integers and $z_1, \dots, z_n \ge 0$ are distinct. Define
$$
V_k := \Big\langle\Big(1 + \frac{u}{k}\Big)^{h_1k}p_1(u), \dots, \Big(1 + \frac{u}{k}\Big)^{h_Nk}p_N(u)\Big\rangle \quad \text{ for } k \ge 0.
$$
Note that for all $k$ sufficiently large (denoted $k\gg 0$), $\big(1 + \frac{u}{k}\big)^{h_ik}p_i(u)$ has degree $h_ik + d_i$ and a zero of order $h_ik$ at $u=-k$. Therefore by \cref{wronskian_degree,wronskian_zero} (with $t=-k$), we can write
$$
\Wr(V_k) = \Big(1 + \frac{u}{k}\Big)^{(h_1 + \cdots + h_N)k + \binom{m_1}{2} + \cdots + \binom{m_l}{2} - \binom{N}{2}}g_k(u) \quad \text{ for all } k \gg 0,
$$
where $g_k\in\C[u]$ is a monic polynomial of degree $n$.

For all $1 \le i \le N$, we have $\lim_{k\to\infty}\big(1 + \frac{u}{k}\big)^{h_ik}p_i(u) = e^{h_iu}p_i(u)$ uniformly on compact subsets of $\C$. Hence $\lim_{k\to\infty}g_k(u) = (u+z_1)\cdots (u+z_n)$, so $g_k(u)$ has distinct nonpositive real zeros for all $k\gg 0$. By \cref{KP_positivity}, $V_k$ is real and totally nonnegative for all $k\gg 0$. Taking $k\to\infty$, we get that $V$ is real and totally nonnegative.\hfill\qed

\begin{rem}
In a similar fashion, one can deduce the reality theorem for quasi-exponentials (\cref{MTV_quasiexp}) from the reality theorem for polynomials (\cref{MTV_SS}).
\end{rem}

\subsection{Proof that \texorpdfstring{\cref{positivity}\ref{positivity_weak}}{Theorem \ref{positivity}\ref{positivity_weak}} implies \texorpdfstring{\cref{positivity}\ref{positivity_strict}}{Theorem \ref{positivity}\ref{positivity_strict}}}\label{poly_proof_strict}
By translation, we may assume that $t=0$. Suppose that part \ref{positivity_weak} of \cref{positivity} is true and we are given $V$ as in part \ref{positivity_strict}. Fix $0 < c < \min(h_1, \dots, h_N)$, and set
$$
W := e^{-cu}\cdot V = \langle e^{(h_1-c)u}p_1(u), \dots, e^{(h_N-c)u}p_N(u)\rangle.
$$
Then the exponents $h_1 - c, \dots, h_N - c$ of $W$ are nonnegative, and
$$
\Wr(W) = e^{-Ncu}\Wr(V) = e^{(h_1 + \cdots + h_N - Nc)u}(u+z_1)\cdots (u+z_n).
$$
Hence $W$ is totally nonnegative about $u=0$ by part \ref{positivity_weak}, so it has a matrix representative $A$ as in \eqref{A_matrix} such that all $N\times N$ minors of $A$ are nonnegative. Moreover, since $z_1, \dots, z_n > 0$, from \eqref{equation_wronskian_empty} we get $\Delta_\varnothing(W) \neq 0$, so the initial $N\times N$ minor of $A$ is positive.

Since $V = e^{cu}\cdot W$, we see that $BA$ is a matrix representative of $V$ about $u=0$, where $B$ is the infinite lower-triangular matrix representing the linear operator $e^{cu}$. Explicitly,
$$
B = \exp\begin{bmatrix}
0 & 0 & 0 & 0 & \cdots \\
c & 0 & 0 & 0 & \cdots \\
0 & 2c & 0 & 0 & \cdots \\
0 & 0 & 3c & 0 & \cdots \\
\vdots & \vdots & \vdots & \vdots & \ddots
\end{bmatrix},
$$
where $\exp(\cdot)$ is the matrix exponential. Then $B$ is triangularly totally positive in the sense that all of its minors which do not trivially vanish by lower-triangularity are positive:
$$
\det(B_{\{i_1, \dots, i_m\}, \{j_1, \dots, j_m\}}) > 0 \quad \text{ for all } m\in\mathbb{N} \text{ and } i_1 \ge j_1, \dots, i_m \ge j_m.
$$
(This is well-known; see \cite[Proof of Lemma 2.29]{karp24}.) Hence by the Cauchy--Binet identity, all $N\times N$ minors of $BA$ are positive, so $V$ is totally positive about $u=0$.\hfill\qed

\section{From quasi-exponentials to polynomials}\label{sec_quasi_to_poly}

\noindent In this section we show that $T_\lambda(u)|_{h_1 = \cdots = h_N = 0} = \beta_\lambda(u)$, where the operators $\beta_\lambda(u)$ were introduced in \cite{karp_purbhoo} and shown to be universal Pl\"{u}cker coordinates for the Wronski map on spaces of polynomials. We then further explore connections with the paper \cite{karp_purbhoo}.

Recall $\alpha_\lambda^{(K)}\in\C[\SG{K}]$ from \eqref{defn_alpha}. Following \cite{karp_purbhoo}, given a partition $\lambda$, we define
\begin{align}\label{defn_beta}
\beta_\lambda(u) := \sum_{\substack{K \subseteq [n], \\ |K| = |\lambda|}}\Big(\prod_{k\in K}(u+z_k)\Big)\,\alpha_\lambda^{(K)} \in\C[\SG{n}]
\end{align}
as a polynomial in $u$. In particular, we may regard $\beta_\lambda(u)$ as an element of $\End((\C^N)^{\otimes n})$.
\begin{example}\label{eg_betalambda}
The operators $\beta_\lambda := \beta_\lambda(0)$ for partitions $\lambda$ of size at most $2$ are as follows:
\begin{align*}
\beta_\varnothing &= z_{[n]}\Sid, & \beta_{(2)} &= \sum_{1 \le k < l \le n}z_{[n]\setminus\{k,l\}}(\Sid + \sigma_{k,l}), \\
\beta_{(1)} &= \sum_{k=1}^n z_{[n]\setminus\{k\}}\Sid, & \beta_{(1,1)} &= \sum_{1 \le k < l \le n}z_{[n]\setminus\{k,l\}}(\Sid - \sigma_{k,l}),
\end{align*}
where $\Sid\in\SG{n}$ denotes the identity. Comparing with \cref{eg_Tlambda}, we see that $T_\lambda\eval{h_1 = \cdots = h_N =0} = \beta_\lambda$ when $|\lambda| \le 2$, in agreement with \cref{Tlambda_beta}.
\end{example}

\begin{thm}\label{Tlambda_beta}
For partitions $\lambda$, we have
$$
T_\lambda(u)|_{h_1 = \cdots = h_N = 0} = \beta_\lambda(u) \quad \text{ in } \End((\C^N)^{\otimes n}).
$$
\end{thm}

\begin{proof}
We may assume that $h=0$. Then the operator $h^{(L\setminus K)}$ is the identity if $K = L$ and zero if $K\subsetneq L$. Therefore in the formula \eqref{Tlambda_equation} for $T_\lambda(u)$, we only need to sum over the terms with $L = K$, whence we obtain precisely the formula \eqref{defn_beta} for $\beta_\lambda(u)$.
\end{proof}

\begin{rem}\label{KP_problem}
We mention that \cite[Problem 5.4]{karp_purbhoo} poses the problem of finding an explicit formula for the element of the Bethe algebra $\betheh \subseteq U(\gl_N[t])$ which coincides with the operator $\beta_\lambda(u)$ when restricted to $\C^N(-z_1) \otimes \cdots \otimes \C^N(-z_n)$. Recall $B_i(u)\in\betheh$ from \eqref{defn_universal}, and by \cref{Tlambda_beta} we have $T_\lambda(u)\eval{h_1 = \cdots = h_N = 0} = \beta_\lambda(u)$. Hence \eqref{equation_ALTZ_determinant} (via \eqref{equation_equality_columns}) provides a formula for the desired element of $\betheh$, but it is not very explicit. We leave it as an open problem to find a more satisfying formula.
\end{rem}

We may also regard $\beta_\lambda(u)$ as an element of $\End(\C[\SG{n}])$, acting by left multiplication. We then obtain the following result proved in \cite{karp_purbhoo}. (The result therein is stated for $t=0$; the case of general $t\in\C$ follows straightforwardly by translation.)
\begin{cor}{\cite[Theorem 1.3(v)]{karp_purbhoo}}\label{KP_coordinates}
Let $z_1, \dots, z_n\in \C$, and fix $t\in\C$.
\begin{enumerate}[label=(\roman*), leftmargin=*, itemsep=2pt]
\item\label{KP_coordinates_forward} If $E\subseteq\C[\SG{n}]$ is a common eigenspace of the operators $\beta_\lambda(t)$, and $\ell(\lambda) \le N$ for all partitions $\lambda$ such that the corresponding eigenvalue of $\beta_\lambda(t)$ is nonzero, then such eigenvalues are the Pl\"{u}cker coordinates about $u=t$ of a space $V_E\in\Gr(N,\C[u])$ of polynomials with Wronskian $(u+z_1) \cdots (u+z_n)$.
\item\label{KP_coordinates_backward} Every space of polynomials $V\in\Gr(N,\C[u])$ with Wronskian $(u+z_1)\cdots(u+z_n)$ equals $V_E$ for some eigenspace $E$, as in part \ref{KP_coordinates_forward}.
\end{enumerate}
\end{cor}

\begin{proof}
This follows from \cref{quasiexp_coordinates,Tlambda_beta}.
\end{proof}

\begin{rem}
Let us explain in more detail how the arguments in this paper and in \cite{alexandrov_leurent_tsuboi_zabrodin14} are connected to \cite{karp_purbhoo}. In \cite{karp_purbhoo}, the positivity theorem for spaces of polynomials (\cref{KP_positivity}) is deduced directly from \cref{KP_coordinates}, using the fact that the operators $\beta_\lambda(t)$ are positive semidefinite when $t \ge -z_1, \dots, -z_n$ (which follows from \cref{alpha_semidefinite}). The proof of \cref{KP_coordinates} in \cite{karp_purbhoo} used three key properties of the operators $\beta_\lambda(u)$ \cite[Theorem 1.3, parts (i)--(iii)]{karp_purbhoo}:
\begin{enumerate}[label=(\roman*), leftmargin=*, itemsep=2pt]
\item\label{betalambda_1} (commutation relations) $\beta_\lambda(u)\beta_\mu(v) = \beta_\mu(v)\beta_\lambda(u)$ for all partitions $\lambda,\mu$;
\item\label{betalambda_2} (translation identity) the $\beta_\lambda(u)$'s satisfy the formula \eqref{translation_identity};
\item\label{betalambda_3} (Pl\"{u}cker relations) the $\beta_\lambda(u)$'s satisfy the quadratic \defn{Pl\"{u}cker relations} which cut out the Grassmannian as a projective variety (see \cite[Theorem 4.1]{carrell_goulden10} for the equations).
\end{enumerate}

In \cite{karp_purbhoo}, the commutation relations \ref{betalambda_1} were proved using the combinatorics of permutations, relying on results of Purbhoo \cite{purbhoo23}. We obtain a new proof of \ref{betalambda_1} as a consequence of \cref{Tlambda_beta} and the commutation relations \eqref{Tlambda_commute} for the operators $T_\lambda(u)$.

In \cite{karp_purbhoo} the translation identity \ref{betalambda_2} was proved using the representation theory of symmetric groups. The operators $T_\lambda(u)$ satisfy a similar identity, which is a consequence of \eqref{translation_identity}, \cref{quasiexp_coordinates}, and \cref{bethe_properties}. The reason we do not need this translation identity in order to prove \cref{quasiexp_coordinates} is that we instead use the dual Jacobi--Trudi identity.

Finally, in \cite{karp_purbhoo} the Pl\"{u}cker relations \ref{betalambda_3} were proved by equivalently showing that the symmetric function $\sum_{\lambda}\beta_\lambda(u) s_\lambda(\x)$ is a $\tau$-function of the KP hierarchy, using the combinatorics of symmetric functions. We obtain a new proof of \ref{betalambda_2} using the result from \cite[Section 4]{alexandrov_leurent_tsuboi_zabrodin14} that the operators $T_\lambda(u)$ satisfy the Pl\"{u}cker relations. The proof in \cite{alexandrov_leurent_tsuboi_zabrodin14} similarly shows that $\sum_\lambda T_\lambda(u) s_\lambda(\x)$ is a $\tau$-function of the KP hierarchy, but uses different methods. It may be interesting to further explore connections between these two proofs.
\end{rem}

\bibliographystyle{alpha}
\bibliography{ref}

\end{document}